\documentclass[a4paper,11pt]{article}
\usepackage{amsmath}
\usepackage{amsfonts}
\usepackage{amssymb}
\usepackage{graphicx}

\newtheorem{theorem}{Theorem}

\newtheorem{corollary}[theorem]{Corollary}

\newtheorem{definition}[theorem]{Definition}
\newtheorem{example}[theorem]{Example}

\newtheorem{proposition}[theorem]{Proposition}
\newtheorem{remark}[theorem]{Remark}

\newenvironment{proof}[1][Proof]{\noindent\textbf{#1.} }{\ \rule{0.5em}{0.5em}}
\input{tcilatex}

\begin{document}

\title{Fundamentals of bicomplex pseudoanalytic function theory: Cauchy
integral formulas, negative formal powers and Schr\"{o}dinger equations with
complex coefficients}
\author{Hugo M. Campos and Vladislav V. Kravchenko \\
{\small Departamento de Matem\'{a}ticas, CINVESTAV del IPN, Unidad
Queretaro, }\\
{\small Libramiento Norponiente No. 2000, Fracc. Real de Juriquilla,
Queretaro, }\\
{\small Qro. C.P. 76230 MEXICO e-mail: hugomcampos@hotmail.com;}\\
{\small vkravchenko@math.cinvestav.edu.mx\thanks{%
Research was supported by CONACYT, Mexico. Hugo Campos additionally
acknowledges the support by FCT, Portugal. }}}
\maketitle

\begin{abstract}
The study of the Dirac system and second-order elliptic equations with
complex-valued coefficients on the plane naturally leads to bicomplex
Vekua-type equations \cite{KrAntonio}, \cite{KrJPhys06}, \cite{CKM AACA}. To
the difference of complex pseudoanalytic (or generalized analytic) functions 
\cite{Bers Book}, \cite{Vekua} the theory of bicomplex pseudoanalytic
functions has not been developed. Such basic facts as, e.g., the similarity
principle or the Liouville theorem in general are no longer available due to
the presence of zero divisors in the algebra of bicomplex numbers.

In the present work we develop a theory of bicomplex pseudoanalytic formal
powers analogous to the developed by L. Bers \cite{Bers Book} and especially
that of negative formal powers. Combining the approaches of L. Bers \ and I.
N. Vekua with some additional ideas we obtain the Cauchy integral formula in
the bicomplex setting. In the classical complex situation this formula was
obtained under the assumption that the involved Cauchy kernel is global, a
very restrictive condition taking into account possible practical
applications, especially when the equation itself is not defined on the
whole plane. We show that the Cauchy integral formula remains valid with the
Cauchy kernel from a wider class called here the reproducing Cauchy kernels.
We give a complete characterization of this class. To our best knowledge
these results are new even for complex Vekua equations. We establish that
reproducing Cauchy kernels can be used to obtain a full set of negative
formal powers for the corresponding bicomplex Vekua equation and present an
algorithm which allows one their construction.

Bicomplex Vekua equations of a special form called main Vekua equations are
closely related to stationary Schr\"{o}dinger equations with complex-valued
potentials. We use this relation to establish useful connections between the
reproducing Cauchy kernels and the fundamental solutions for the Schr\"{o}%
dinger operators which allow one to construct the Cauchy kernel when the
fundamental solution is known and vice versa. Moreover, using these results
we construct the fundamental solutions for the Darboux transformed Schr\"{o}%
dinger operators.
\end{abstract}

\section{Introduction}

In the present work we study the bicomplex Vekua equations of the form 
\begin{equation}
\partial _{\overline{z}}W=aW+b\overline{W}  \label{VekuabicIntro}
\end{equation}%
where $a$, $b$ and $W$ are functions of the complex variable $z=x+jy$ and
take values in the algebra of bicomplex numbers. The conjugation $\overline{W%
}$ is with respect to the imaginary unit $j$ and $\partial _{\overline{z}}=%
\frac{1}{2}\left( \partial _{x}+j\partial _{y}\right) $. Every bicomplex
Vekua equation (\ref{VekuabicIntro}) is equivalent to a first order system 
\begin{align}
\partial _{x}u-\partial _{y}v& =\alpha u+\beta v  \label{s1} \\
\partial _{x}v+\partial _{y}u& =\gamma u+\delta v  \label{s2}
\end{align}%
where all the involved functions are complex, and vice versa, the system can
be written in the form (\ref{VekuabicIntro}). As always, whenever it is
possible, the introduction of an appropriate algebraic structure leads to a
deeper understanding of the system of equations and this is why it is
preferable to study (\ref{s1}), (\ref{s2}) in the form (\ref{VekuabicIntro}%
). Equation (\ref{VekuabicIntro}) arised in \cite{KrAntonio} in relation
with the Dirac system with electromagnetic and scalar potentials in the
two-dimensional case. In \cite{KrJPhys06} it was shown that equation (\ref%
{VekuabicIntro}) when $a\equiv 0$ and $b=\partial _{\overline{z}}f/f$ where $%
f$ is a scalar function (see Section \ref{Sect Bicomplex numbers}) is
closely related to the stationary Schr\"{o}dinger equation 
\begin{equation}
\Delta u=qu  \label{sch_intro}
\end{equation}%
with $q=\Delta f/f$. Vekua equations with coefficients $a$ and $b$ of this
special form are called Vekua equations of the main type or main Vekua
equations \cite{APFT}. In the classical complex case they are closely
related to so-called $p$-analytic functions (for the theory of $p$-analytic
functions we refer to the book \cite{Polozhy}, for the relation to the main
Vekua equation to \cite[Chapter 5]{APFT} and for their applications to \cite%
{AstalaPaivarinta}, \cite{Fischer}, \cite{KrPanalyt}, \cite{APFT}, \cite%
{Ramirez et al 2010}, \cite{ZabarankinKrokhmal2007}, \cite%
{ZabarankinUlitko2006}).

The relation between the main Vekua equation and the stationary Schr\"{o}%
dinger equation is of the same nature as the relation between the
Cauchy-Riemann system and the Laplace equation. The scalar part of the
solution of the main Vekua equation is necessarily a solution of the
corresponding Schr\"{o}dinger equation and vice versa, for any solution $u$
of the Schr\"{o}dinger equation its \textquotedblleft conjugate
metaharmonic\textquotedblright\ counterpart $v$ (see \cite[Chapter 3]{APFT})
can be constructed such that the obtained bicomplex function $W=u+jv$ will
be a solution of the main Vekua equation. The obtained counterpart in its
turn is a solution of another Schr\"{o}dinger equation the potential of
which is a Darboux transformation of the initial potential $q$. In \cite%
{Krpseudoan} (see also \cite{APFT}) the procedure for construction of $v$ by 
$u$ and vice versa was obtained in the explicit form. Recently Sh. Garuchava
in \cite{Garuchava2011} established a correspondence between the procedure
from \cite{Krpseudoan} and the two-dimensional Darboux transformation from 
\cite{MS}.

The theory of bicomplex Vekua equations is far from being complete though
publications attempting to extend the essential properties of pseudoanalytic
(or generalized analytic) functions onto the solutions of more general
systems on the plane are numerous. We refer to \cite{Youvaraj Jain 1990} for
some first results in studying (\ref{VekuabicIntro}) and emphasize that such
important facts as the similarity principle were obtained in that paper
under the overrestrictive condition on the coefficients ($b\equiv 0$). Under
this condition the bicomplex Vekua equation (\ref{VekuabicIntro}) reduces to
a pair of decoupled complex Vekua equations which obviously simplifies its
study. In \cite{Rochon2008} relations between classes of bicomplex Vekua
equations and complexified Schr\"{o}dinger equations were studied. In the
recent work \cite{Berglez2010} several results on bicomplex pseudoanalytic
functions from the point of view of Bauer-Peschl \ differential operators
can be found.

The main difficulty in studying the bicomplex Vekua equation comes from the
existence of zero divisors in the algebra of bicomplex numbers. For example,
when the solution $W$ of (\ref{VekuabicIntro}) does not have zero divisors
(the values of the function do not coincide with a zero divisor at any point 
$z$ of the domain of interest) for such $W$ we prove a similarity principle
(Theorem \ref{Similarity. P} below). Unfortunately in general this fact is
unavailable which forces one to look for alternative ideas for developing
the corresponding pseudoanalytic function theory.

In \cite{CKM AACA}, \cite{KrAntonio}, \cite{KrJPhys06}, \cite{APFT} it was
noticed that several constructive results from Bers' theory remain valid in
the bicomplex case. For example, if a generating sequence corresponding to
the bicomplex Vekua equation (\ref{VekuabicIntro}) is known, the
construction of corresponding positive formal powers can be performed by
means of Bers' algorithm based on the concept of the $(F,G)$-integration.
However to the difference from the complex pseudoanalytic function theory it
is not clear how to prove the expansion and Runge theorems in the bicomplex
case, the results ensuring the completeness of the formal powers in the
space of all solutions of the Vekua equation. Moreover, these results in the
classical complex case were obtained for global formal powers only (see the
definition in Subsection \ref{Subsect Formal powers}) meanwhile the only
system of formal powers whose explicit form is known corresponds to the case
of analytic functions, $\left\{ \left( z-z_{0}\right) ^{n}\right\}
_{0}^{\infty }$. Recently for a wide class of main bicomplex Vekua equations
the expansion and the Runge theorems were obtained in \cite{CKM AACA}, \cite%
{CKT 2012} using the approach based on so-called transmutation (or
transformation) operators. It is important to emphasize that in \cite{CKM
AACA} and \cite{CKT 2012} the results concerning the completeness of systems
of \textbf{local} formal powers were obtained which opened the way to apply
them in practical solution of boundary value and eigenvalue problems for
second-order elliptic equations with variable coefficients (see \cite%
{CCK2012} and \cite{CKR2011}).

The main subject of the present work is the study of negative formal powers
for the bicomplex Vekua equation and their applications. Based on the
results on the Cauchy kernels we obtain Cauchy integral formulas for
bicomplex pseudoanalytic functions. In the classical theory of complex
pseudoanalytic functions in fact there are two types of Cauchy integral
formulas. One is based on Cauchy kernels for an adjoint Vekua equation and
the other involves the Cauchy kernels for the initial Vekua equation (here
we call these two Cauchy integral formulas the first and the second
respectively). We obtain a relation between both Cauchy kernels and prove
both kinds of Cauchy integral formulas. It is important to mention that the
Cauchy kernels involved in the obtained Cauchy integral formulas are not
required to be global but instead belong to a much more general class which
we call reproducing Cauchy kernels. We give a complete characterization of
this class. To our best knowledge these results are new even for complex
Vekua equations. We establish that reproducing Cauchy kernels can be used to
obtain a full set of negative formal powers for the corresponding bicomplex
Vekua equation and present an algorithm which allows one their construction.

Bicomplex Vekua equations of a special form called main Vekua equations are
closely related to stationary Schr\"{o}dinger equations with complex-valued
potentials. We use this relation to establish direct connections between the
reproducing Cauchy kernels and the fundamental solutions for the Schr\"{o}%
dinger operators which allow one to construct the Cauchy kernel when the
fundamental solution is known and vice versa. Moreover, using these results
we construct the fundamental solutions for the Darboux transformed Schr\"{o}%
dinger operators. These results are also new in the context of the classical
complex Vekua equations and among other applications allow one to obtain in
a closed form a reproducing Cauchy kernel and a set of negative formal
powers for an important class of main Vekua equations.

The layout of the paper is as follows. In Section \ref{Sect Bicomplex
numbers} we introduce the necessary formalism concerning the bicomplex
numbers. Some of the facts presented here can be found in several sources,
e.g., \cite{KSbook}, \cite{RochonShapiro}, \cite{RochonTrembl}. Nevertheless
we needed to introduce a special though quite natural norm and hence prove
several related properties which probably are first published. In Section %
\ref{Sect Bicomplex pseudoanalytic} we introduce the necessary definitions
from pseudoanalytic function theory, prove several results concerning
bicomplex pseudoanalytic functions, like the mentioned above similarity
principle (for functions without zero divisors). By analogy with the complex
case we define formal powers and obtain their basic properties. We introduce
the main Vekua equation in relation with the stationary Schr\"{o}dinger
equation and obtain some auxiliary results for its solutions. In Section \ref%
{Sect Cauchy Int Formulas} the Cauchy integral formulas for bicomplex
pseudoanalytic functions are obtained and the characterization of the
reproducing Cauchy kernels is given. In Section \ref{Sect Negative Formal
Powers} we establish an important relation between the negative formal
powers corresponding to different bicomplex Vekua equations which in fact
leads to an algorithm for constructing the negative formal powers. We give
several applications of this result in Section \ref{Sect Applications} using
a valuable observation that for the main Vekua equation the adjoint and the
successor coincide. This leads to the possibility to construct a reproducing
Cauchy kernel for the Vekua equation from a known fundamental solution for a
related Schr\"{o}dinger equation and vice versa and also gives a method for
constructing the fundamental solutions for a chain of Darboux transformed
Schr\"{o}dinger operators. Finally, several examples of explicitly
calculated kernels and fundamental solutions are presented.

\section{Bicomplex numbers\label{Sect Bicomplex numbers}}

Together with the imaginary unit $i$ we consider another imaginary unit $j$,
such that 
\begin{equation}
j^{2}=i^{2}=-1\quad\text{and}\quad i\,j=j\,i.  \label{laws}
\end{equation}
We have then two copies of the field of complex numbers, $\mathbb{C}%
_{i}:=\left\{ a+ib,\text{ }\left\{ a,b\right\} \subset\mathbb{R}\right\} $
and $\mathbb{C}_{j}:=\left\{ a+jb,\text{ }\left\{ a,b\right\} \subset\mathbb{%
R}\right\} $. The expressions of the form $W=u+jv$ where $\left\{
u,v\right\} \subset\mathbb{C}_{i}$ are called bicomplex numbers. The
conjugation with respect to $j$ we denote as $\overline{W}=u-jv$. The
components $u$ and $v$ will be called the scalar and the vector part of $W$
respectively. We will use the notation $u=\func{Sc}W$ and $v=\func{Vec}W$.

The set of all bicomplex numbers with a natural operation of addition and
with the multiplication defined by the laws (\ref{laws}) represents a
commutative ring with a unit. We denote it by $\mathbb{B}$. An element $W\in%
\mathbb{B}$ is invertible if and only if $W\overline{W}\neq0$ and the
inverse element \ is defined by the equality $W^{-1}=\dfrac{\overline{W}}{%
\overline{W}W}$.

Let $\mathcal{R}(\mathbb{B)}$ be the set formed by the invertible elements
of $\mathbb{B}$ and $\sigma(\mathbb{B)}$ denote the generalized zeros of $%
\mathbb{B}$ (zero divisors), that is%
\begin{equation*}
\sigma(\mathbb{B)=}\left\{ W\in\mathbb{B}\text{: }W\neq0\text{ and }W%
\overline{W}=0\right\} .
\end{equation*}

It is convenient to introduce the pair of idempotents $P^{+}=\frac{1}{2}%
(1+ij)$ and $P^{-}=\frac{1}{2}(1-ij)$ ($\left( P^{\pm}\right) ^{2}=P^{\pm}$%
). As it can be verified directly, $P^{\pm}\in\sigma(\mathbb{B)}$ and $%
P^{+}+P^{-}=1$.

\begin{proposition}
Let $W\in \mathbb{B}$.Then

\begin{description}
\item[$(i)$] there exist the unique numbers $W^{+}$, $W^{-}\in \mathbb{C}%
_{i} $ such that $W=P^{+}W^{+}+P^{+}W^{-}$ which can be computed from $W$ as
follows%
\begin{equation}
W^{\pm }=\func{Sc}W\mp i\limfunc{Vec}W\text{,}  \label{W+-}
\end{equation}

\item[$(ii)$] a nonzero element $W\ $belongs to $\sigma (\mathbb{B)}$ iff $%
W=P^{+}W^{+}$ or $W=P^{-}W^{-}$.
\end{description}
\end{proposition}

\begin{proof}
$(i)$ Straightforward.

$(ii)$ It follows directly from the definition of $\sigma (\mathbb{B)}$ and
from the equality%
\begin{equation*}
W\overline{W}=W^{+}W^{-}\text{.}
\end{equation*}
\end{proof}

For $W=P^{+}W^{+}+P^{-}W^{-}\in\mathbb{B}$ we introduce the notation 
\begin{equation}
\left\vert W\right\vert =\frac{1}{2}\left( \left\vert W^{+}\right\vert _{%
\mathbb{C}_{i}}+\left\vert W^{-}\right\vert _{\mathbb{C}_{i}}\right) ,
\label{bicom_norm}
\end{equation}
where $\left\vert \cdot\right\vert _{\mathbb{C}_{i}}$ is the usual norm in $%
\mathbb{C}_{i}$.

\begin{proposition}
\label{Prop_norm}The function defined in (\ref{bicom_norm}) is a norm in $%
\mathbb{B}$ and possesses the following properties.

\begin{description}
\item[$(i)$] If $W\in\mathbb{C}_{i}$ (that is, $VecW=0$) then $\left\vert
W\right\vert =\left\vert W\right\vert _{\mathbb{C}_{i}}$.

\item[$(ii)$] If $W\in\mathbb{C}_{j}$ and $V\in\mathbb{B}$ then $\left\vert
W\right\vert =\left\vert W\right\vert _{\mathbb{C}_{j}}$ and $\left\vert
WV\right\vert =\left\vert W\right\vert \left\vert V\right\vert $.

\item[$(iii)$] If $W$, $V\in\mathbb{B}$ then $\left\vert WV\right\vert
\leq2\left\vert W\right\vert \left\vert V\right\vert $ and%
\begin{equation*}
\left\vert \limfunc{Sc}W\right\vert \leq\left\vert W\right\vert
,\quad\left\vert \limfunc{Vec}W\right\vert \leq\left\vert W\right\vert
\quad,\left\vert W\right\vert \leq\left\vert \limfunc{Sc}W\right\vert
+\left\vert \limfunc{Vec}W\right\vert .
\end{equation*}
\end{description}
\end{proposition}

\begin{proof}
The proof of $(i)$, $(ii)$ and of the last part of $(iii)$ is
straightforward. Let $W=P^{+}W^{+}+P^{+}W^{-}$ and $V=P^{+}V^{+}+P^{+}V^{-}%
\in\mathbb{B}$. Then%
\begin{equation*}
WV=P^{+}W^{+}V^{+}+P^{-}W^{-}V^{-}
\end{equation*}
and%
\begin{equation*}
\left\vert WV\right\vert =\frac{1}{2}\left( \left\vert W^{+}V^{+}\right\vert
+\left\vert W^{-}V^{-}\right\vert \right) \leq\frac{1}{2}\left( \left\vert
W^{+}\right\vert +\left\vert W^{-}\right\vert \right) \left( \left\vert
V^{+}\right\vert +\left\vert V^{-}\right\vert \right) =2\left\vert
W\right\vert \left\vert V\right\vert .
\end{equation*}
\end{proof}

\begin{proposition}
\label{R(B)_open_prop}$\mathcal{R}(\mathbb{B)}$ is open in $\mathbb{B}.$
\end{proposition}

\begin{proof}
Let $W\in \mathcal{R}(\mathbb{B)}$. Then $W=P^{+}W^{+}+P^{-}W^{-}$ where $%
W^{+}\neq 0$ and $W^{-}\neq 0$. Define $r=\min \left\{ \frac{\left\vert
W^{+}\right\vert }{2}\text{,}\frac{\left\vert W^{-}\right\vert }{2}\right\} $
and consider the open ball $B(W,r)$. If $V\in \left( \mathcal{R}(\mathbb{B)}%
\right) ^{c}\mathbb{=}\sigma (\mathbb{B)\cup }\left\{ 0\right\} $ then $%
V=P^{+}V^{+}$ for which we have 
\begin{equation*}
\left\vert W-V\right\vert =\frac{1}{2}\left( \left\vert
W^{+}-V^{+}\right\vert +\left\vert W^{-}\right\vert \right) \geq \frac{%
\left\vert W^{-}\right\vert }{2}\geq r
\end{equation*}%
or $V=P^{-}V$ for which we have%
\begin{equation*}
\left\vert W-V\right\vert =\frac{1}{2}\left( \left\vert W^{+}\right\vert
+\left\vert W^{-}-V^{-}\right\vert \right) \geq \frac{\left\vert
W^{+}\right\vert }{2}\geq r.
\end{equation*}%
Thus, $V\notin B(W,r)$ and hence $B(W,r)\subset \mathcal{R(}\mathbb{B)}$
which proves that $\mathcal{R}(\mathbb{B)}$ is open.
\end{proof}

An exponential function of a bicomplex variable is defined by the equality 
\begin{equation*}
E\left[ W\right] :=P^{+}e^{W^{+}}+P^{-}e^{W^{-}},\quad W\in\mathbb{B}.
\end{equation*}

\begin{proposition}
\label{Ez_prop}

\begin{description}
\item[$(i)$] $E\left[ W+V\right] =E\left[ W\right] E\left[ V\right] $, $%
\forall\,W$, $V\in\mathbb{B}$, in particular, $E\left[ W\right] $ is
invertible with the inverse given by $E\left[ -W\right] $;

\item[$(ii)$] $\left\vert E\left[ W\right] -1\right\vert \leq e^{2\left\vert
W\right\vert }-1$ \ for all $\,W\in\mathbb{B}$.
\end{description}
\end{proposition}

\begin{proof}
$(i)$ Straightforward.

$(ii)$ Follows from the equality%
\begin{equation*}
E\left[ W\right] -1=P^{+}\left( e^{W^{+}}-1\right) +P^{-}\left(
e^{W^{-}}-1\right)
\end{equation*}
and from the fact that $\left\vert e^{W^{\pm}}-1\right\vert \leq
e^{\left\vert W^{\pm}\right\vert }-1$.
\end{proof}

\section{Bicomplex pseudoanalytic functions\label{Sect Bicomplex
pseudoanalytic}}

\subsection{Generating pair and first properties of bicomplex pseudoanalytic
functions}

\bigskip

\begin{definition}
A pair of $\mathbb{B}$-valued functions $F$ and $G$ possessing H\"{o}lder
continuous partial derivatives in $\Omega\subset\mathbb{C}_{j}$ with respect
to the real variables $x$ and $y$ is said to be a generating pair if it
satisfies the inequality%
\begin{equation}
\func{Vec}(\overline{F}G)\neq0\qquad\text{in }\Omega.  \label{condGenPair}
\end{equation}
\end{definition}

Condition (\ref{condGenPair}) implies that every bicomplex function $W$
defined in a subdomain of $\Omega$ admits the unique representation $W=\phi
F+\psi G$ where the functions $\phi$ and $\psi$ are scalar (=$\mathbb{C}_{i}$%
-valued).

\bigskip

Assume that $(F,G)$ is a generating pair in a domain $\Omega$.

\begin{definition}
Let the $\mathbb{B}$-valued function $W$ be defined in a neighborhood of $%
z_{0}\in\Omega$. In a complete analogy with the complex case we say that at $%
z_{0}$ the function $W$ possesses the $(F,G)$-derivative $\overset{\cdot}{W}%
(z_{0})$ if the (finite) limit 
\begin{equation}
\overset{\circ}{W}(z_{0})=\lim_{z\rightarrow z_{0}}\frac{W(z)-\lambda
_{0}F(z)-\mu_{0}G(z)}{z-z_{0}}  \label{derivative_def}
\end{equation}
exists where $\lambda_{0}$ and $\mu_{0}$ are the unique scalar constants
such that $W(z_{0})=\lambda_{0}F(z_{0})+\mu_{0}G(z_{0})$.
\end{definition}

We will also use the notation $\frac{d_{(F,G)}W(z_{0})}{dz}=\overset{\circ}{W%
}(z_{0})$.

Let us introduce the bicomplex operators 
\begin{equation}
\partial_{\overline{z}}=\frac{1}{2}\left( \partial_{x}+j\partial_{y}\right)
,\quad\partial_{z}=\frac{1}{2}\left( \partial_{x}-j\partial_{y}\right)
\label{CR_Op_Bicomp}
\end{equation}
acting on $\mathbb{B}$-valued functions. Similarly to the complex case (see,
e.g., \cite[Chapter 2]{APFT}) it is easy to show that if $\overset{\circ}{W}%
(z_{0})$ exists then at $z_{0}$, $\partial_{\overline{z}}W$ and $%
\partial_{z}W$ exist and the equations 
\begin{equation}
\partial_{\overline{z}}W=aW+b\overline{W}  \label{Vekua_equation (F,G)}
\end{equation}
and 
\begin{equation}
\overset{\circ}{W}=\partial_{z}W-AW-B\overline{W}
\label{derivative_with_characteristic}
\end{equation}
hold, where $a$, $b$, $A$ and $B$ are the \emph{characteristic coefficients}
of the pair $(F,G)$ defined by the formulas 
\begin{equation*}
a=a_{(F,G)}=-\frac{\overline{F}\,\partial_{\overline{z}}G-\overline {G}%
\,\partial_{\overline{z}}F}{F\overline{G}-\overline{F}G},\qquad b=b_{(F,G)}=%
\frac{F\,\partial_{\overline{z}}G-G\,\partial_{\overline{z}}F}{F\overline{G}-%
\overline{F}G},
\end{equation*}

\begin{equation*}
A=A_{(F,G)}=-\frac{\overline{F}\,\partial_{z}G-\overline{G}\,\partial_{z}F}{F%
\overline{G}-\overline{F}G},\qquad B=B_{(F,G)}=\frac{F\,\partial
_{z}G-G\,\partial_{z}F}{F\overline{G}-\overline{F}G}.
\end{equation*}
Notice that $F\overline{G}-\overline{F}G=-2j\func{Vec}(\overline {F}G)\neq0$.

If $\partial_{\overline{z}}W$ and $\partial_{z}W$ exist and are continuous
in some neighborhood of $z_{0}$, and if (\ref{Vekua_equation (F,G)}) holds
at $z_{0}$, then $\overset{\circ}{W}(z_{0})$ exists, and (\ref%
{derivative_with_characteristic}) holds. Let us notice that $F$ and $G$
possess $(F,G)$-derivatives, $\overset{\circ}{F}\equiv\overset{\circ}{G}%
\equiv0$, and the following equalities are valid which determine the
characteristic coefficients uniquely%
\begin{equation*}
\partial_{\overline{z}}F=aF+b\overline{F},\quad\partial_{\overline{z}}G=aG+b%
\overline{G},
\end{equation*}%
\begin{equation*}
\partial_{z}F=AF+B\overline{F},\quad\partial_{z}G=AG+B\overline{G}.
\end{equation*}

\begin{definition}
A function will be called $\mathbb{B}$-$(F,G)$-pseudoanalytic (or, simply, $%
\mathbb{B}$-pseudoanalytic, if there is no danger of confusion) in a domain $%
\Omega$ if $\overset{\circ}{W}(z)$ exists everywhere in $\Omega$.
\end{definition}

\begin{remark}
When $F\equiv1$ and $G\equiv j$ the corresponding bicomplex Vekua equation
is 
\begin{equation}
\partial_{\overline{z}}W=0,  \label{C-Rbic}
\end{equation}
and its study in fact reduces to the complex analytic function theory. To
see this notice that with the aid of the idempotents $P^{\pm}$ the operator $%
\partial_{\overline{z}}$ admits the following representation%
\begin{equation}
\partial_{\overline{z}}=P^{+}d_{z}+P^{-}d_{\overline{z}}  \label{CR_op_Fact}
\end{equation}
where 
\begin{equation*}
d_{z}=\frac{1}{2}(\partial_{x}-i\partial_{y})\text{\quad}d_{\overline{z}}=%
\frac{1}{2}(\partial_{x}+i\partial_{y})
\end{equation*}
are the complex Cauchy-Riemann operators. Then the function $%
W=P^{+}W^{+}+P^{-}W^{-}$ satisfies (\ref{C-Rbic}) iff the scalar functions $%
W^{+}$ and $W^{-}$ (defined by (\ref{W+-})) are antiholomorphic and
holomorphic respectively.
\end{remark}

\begin{remark}
\label{remark_decouple_vekua}The bicomplex Vekua equation (\ref%
{Vekua_equation (F,G)}) is equivalent to the following first order elliptic
system%
\begin{align*}
\partial _{x}u-\partial _{y}v& =(\alpha +\theta )u+\left( \gamma -\beta
\right) v \\
\partial _{x}v+\partial _{y}u& =\left( \beta +\gamma \right) u+\left( \alpha
-\theta \right) v
\end{align*}%
where $u=\limfunc{Sc}W$, $v=\limfunc{Vec}W$, $\alpha =2\limfunc{Sc}a$, $%
\beta =2\limfunc{Vec}a$, $\theta =2\limfunc{Sc}b$, $\gamma =2\limfunc{Vec}b$
are $\mathbb{C}_{i}$-valued functions. We stress that if $a$ and $b$ are $%
\mathbb{C}_{j}$-valued functions then $\alpha $, $\beta $, $\theta $ and $%
\gamma $ are real valued functions and the bicomplex Vekua equation (\ref%
{Vekua_equation (F,G)}) can be decoupled into a pair of independent complex
Vekua equations. For example, $W$ is a solution of (\ref{Vekua_equation
(F,G)}) iff $H_{1}:=\func{Re}\limfunc{Sc}W+j\limfunc{Re}\limfunc{Vec}W$ and $%
H_{2}:=\func{Im}\limfunc{Sc}W+j\func{Im}\limfunc{Vec}W$ satisfy%
\begin{equation}
\partial _{\overline{z}}H_{1}=aH_{1}+b\overline{H_{1}},\quad \partial _{%
\overline{z}}H_{2}=aH_{2}+b\overline{H_{2}}.  \label{decoupledvekua}
\end{equation}%
Notice that all the functions involved in (\ref{decoupledvekua}) are $%
\mathbb{C}_{j}$-valued which means that the theory developed by L. Bers and
I. Vekua on complex Vekua equations (\cite{Bers Book}, \cite{Vekua}, \cite%
{APFT}) can be applied to study (\ref{decoupledvekua}) and therefore to the
study of (\ref{Vekua_equation (F,G)}) in this special case. In general the
reduction of a bicomplex Vekua equation to a pair of decoupled complex Vekua
equations is not possible.
\end{remark}

Let $\Omega \subset \mathbb{R}^{2}$ and $\mathcal{A}$, $\mathcal{B}$ be the
operators acting on complex functions by the rules 
\begin{equation*}
(\mathcal{A}\Phi \mathcal{)(}z)\mathcal{=}\frac{1}{\pi }\int_{\Omega }\frac{%
\Phi (z)}{z-\zeta }d\Omega _{\zeta }
\end{equation*}%
and%
\begin{equation*}
\mathcal{B=CAC}
\end{equation*}%
where $\mathcal{C}$ is the operator of complex conjugation (with respect to $%
i$). We recognize immediately that $\mathcal{A}$ and $\mathcal{B}$ are the
right inverse operators of $d_{\overline{z}}$ and $d_{z}$ respectively.
Consider the operator%
\begin{equation*}
T_{\overline{z}}W=P^{+}\mathcal{B}W^{+}+P^{-}\mathcal{A}W^{-}
\end{equation*}%
acting on bicomplex functions.

\begin{proposition}
\label{Tz_prop}Let $W$ be a bounded measurable $\mathbb{B}$-valued function
defined in some bounded domain $\Omega\subset\mathbb{C}_{j}$.

\begin{description}
\item[$(i)$] If $\left\vert W(z)\right\vert \leq M$ in $\Omega$ then $%
\left\vert T_{\overline{z}}W(z)\right\vert \leq k_{1}M$ for all $z\in\Omega$%
, $k_{1}$ depending only on the area of $\Omega$.

\item[$(ii)$] For every $z_{1}$, $z_{2}\in\mathbb{C}_{j}$,%
\begin{equation*}
\left\vert T_{\overline{z}}W(z_{1})-T_{\overline{z}}W(z_{2})\right\vert \leq
k_{2}M\left\vert z_{1}-z_{2}\right\vert \left\{ 1+\log^{+}\frac{1}{%
\left\vert z_{1}-z_{2}\right\vert }\right\}
\end{equation*}
where $k_{2}$ depends only on the diameter of $\Omega$ and $%
\log^{+}\alpha=\log\alpha$ if $\alpha>1$, and $\log^{+}\alpha=0$ if $%
\alpha\leq1$.

\item[$(iii)$] In every connected component of the complement of $\overline{%
\Omega}$, $T_{\overline{z}}W(z)$ is a bicomplex analytic function.

\item[$(iv)$] If $W$ satisfies the H\"{o}lder condition at a point $%
z_{0}\in\Omega$ then $\partial_{\overline{z}}T_{\overline{z}}W(z_{0})$ and $%
\partial_{z}T_{\overline{z}}W(z_{0})$ exist, and $\partial_{\overline{z}}T_{%
\overline{z}}W(z_{0})=W(z_{0})$.

\item[$(v)$] If $W$ is H\"{o}lder continuous in $\Omega$, then so are $%
\partial_{\overline{z}}T_{\overline{z}}W$ and $\partial_{z}T_{\overline{z}}W$%
.
\end{description}
\end{proposition}

\begin{proof}
The proof follows directly from the well known properties of the operator $%
\mathcal{A}$ \cite[pag 7]{Bers Book}.
\end{proof}

\bigskip

With the help of the last proposition the following useful results
concerning bicomplex pseudoanalytic functions are obtained.

\begin{proposition}
\label{prop_int_eq}Let $W$ be a bounded measurable $\mathbb{B}$-valued
function in some domain $\Omega$. Set%
\begin{equation*}
h=W-T_{\overline{z}}(aW+b\overline{W}).
\end{equation*}
Then, $W$ is $\mathbb{B}$-pseudoanalytic iff $h$ is $\mathbb{B}$-analytic.
\end{proposition}

\begin{proof}
Suppose that $h$ is $\mathbb{B}$-analytic. Then as $W$ is bounded, by
Proposition \ref{Tz_prop} the function $T_{\overline{z}}(aW+b\overline{W})$
is H\"{o}lder continuous. As $h$ is H\"{o}lder continuous (being $\mathbb{B}$%
-analytic) the function $W$ is H\"{o}lder continuous as well and by
Proposition \ref{Tz_prop} the function $T_{\overline{z}}(aW+b\overline{W})$
is continuously differentiable. As $\partial _{\overline{z}}h=0$,
application of $\partial _{\overline{z}}$ to $W$ gives 
\begin{equation*}
\partial _{\overline{z}}W=\partial _{\overline{z}}h+\partial _{\overline{z}%
}T_{\overline{z}}(aW+b\overline{W})=aW+b\overline{W}.
\end{equation*}%
The other direction of this statement can be proved using similar techniques
as in \cite[p. 8]{Bers Book}.
\end{proof}

\begin{proposition}
A $\mathbb{B}$-pseudoanalytic function in $\Omega$ has H\"{o}lder continuous
partial derivatives in every compact subdomain of $\Omega$.
\end{proposition}

\begin{proof}
The proof follows from the above proposition and from the properties of the
operator $T_{\overline{z}}$ established in Proposition \ref{Tz_prop}.
\end{proof}

\begin{theorem}
\label{theorem_pseudo_conv}The limit of a uniformly convergent sequence of $%
\mathbb{B}$-pseudoanalytic functions is $\mathbb{B}$-pseudoanalytic.
\end{theorem}

\begin{proof}
Let $W_{n}(z)$ be a sequence of bounded $\mathbb{B}$-pseudoanalytic
functions in a bounded domain $\Omega $ such that $W_{n}(z)\rightarrow W(z)$
uniformly in $\Omega $. It follows from Proposition \ref{prop_int_eq} that
the functions $h_{n}$ defined by%
\begin{equation*}
h_{n}=W_{n}-T_{\overline{z}}(aW_{n}+b\overline{W_{n}})
\end{equation*}%
are $\mathbb{B}$-analytic in $\Omega $. Then the uniform limit $h=W-T_{%
\overline{z}}(aW+b\overline{W})$ is also a $\mathbb{B}$-analytic function in 
$\Omega $. The last equality together with Proposition \ref{prop_int_eq}
allows one to conclude that $W$ is $\mathbb{B}$-pseudoanalytic in $\Omega $.
\end{proof}

\begin{theorem}
\label{Similarity. P}Let $W$ be a $\mathbb{B}$-pseudoanalytic function in $%
\Omega$ except perhaps a finite number of points $\left\{
z_{1},..,z_{p}\right\} $ where $W$ is allowed to be unbounded. Then if $%
W^{-1}\overline{W}$ is a bounded measurable function in $\Omega$, there
exists a $\mathbb{B}$-analytic function $\Psi$ in $\Omega\backslash\left\{
z_{1},..,z_{p}\right\} $ such that%
\begin{equation*}
W=\Psi E\left[ S\right] ,\quad\text{in }\Omega\text{,}
\end{equation*}
where $S=T_{\overline{z}}(a+\frac{\overline{W}}{W}b)$.
\end{theorem}

\begin{proof}
Application of $\partial _{\overline{z}}$ to $\Psi =WE\left[ -S\right] $
gives%
\begin{align*}
\partial _{\overline{z}}\Psi & =\left( \partial _{\overline{z}}W\right) E%
\left[ -S\right] +W\partial _{\overline{z}}E\left[ -S\right] \\
& =\left( aW+b\overline{W}\right) E\left[ -S\right] +W\left( -a-\frac{%
\overline{W}}{W}b\right) E\left[ -S\right] =0\text{.}
\end{align*}%
Then $\Psi $ is $\mathbb{B}$-analytic$.$
\end{proof}

\subsection{Vekua's equation for (F,G)-derivatives and the antiderivative}

\begin{definition}
\label{DefSuccessor_bi}Let $(F,G)$ and $(F_{1},G_{1})$ be two generating
pairs in $\Omega$. $(F_{1},G_{1})$ is called a\ successor of $(F,G)$ and $%
(F,G)$ is called a predecessor of $(F_{1},G_{1})$ if%
\begin{equation*}
a_{(F_{1},G_{1})}=a_{(F,G)}\qquad\text{and}\qquad
b_{(F_{1},G_{1})}=-B_{(F,G)}\text{.}
\end{equation*}
\end{definition}

By analogy with the complex case \cite{Bers Book} (see also \cite{APFT}) we
have the following statements

\begin{theorem}
\label{ThBersDer_bi}Let $W$ be a bicomplex $(F,G)$-pseudoanalytic function
and let $(F_{1},G_{1})$ be a successor of $(F,G)$. Then $\overset{\circ}{W}$
is an $(F_{1},G_{1})$-pseudoanalytic function.
\end{theorem}

\begin{definition}
\label{DefSeq_bi}A sequence of generating pairs $\left\{
(F_{m},G_{m})\right\} $, $m=0,\pm1,\pm2,\ldots$ is called a generating
sequence if $(F_{m+1},G_{m+1})$ is a successor of $(F_{m},G_{m})$. If $%
(F_{0},G_{0})=(F,G)$, we say that $(F,G)$ is embedded in $\left\{
(F_{m},G_{m})\right\} $.
\end{definition}

Let $W$ be a bicomplex $(F,G)$-pseudoanalytic function. Using a generating
sequence in which $(F,G)$ is embedded we can define the higher derivatives
of $W$ by the recursion formula%
\begin{equation*}
W^{[0]}=W;\qquad W^{[m+1]}=\frac{d_{(F_{m},G_{m})}W^{[m]}}{dz},\quad
m=0,1,\ldots\text{.}
\end{equation*}

\begin{definition}
\label{DefAdjoint_bi}Let $(F,G)$ be a generating pair. Its adjoint
generating pair $(F,G)^{\ast}=(F^{\ast},G^{\ast})$ is defined by the formulas%
\begin{equation*}
F^{\ast}=-\frac{2\overline{F}}{F\overline{G}-\overline{F}G},\qquad G^{\ast }=%
\frac{2\overline{G}}{F\overline{G}-\overline{F}G}.
\end{equation*}
\end{definition}

\begin{proposition}
\label{cara_coef_(F,G)*}$(F,G)^{\ast\ast}=(F,G)$ and%
\begin{equation*}
\begin{array}{cc}
a_{(F^{\ast},G^{\ast})}=-a_{(F,G)}\quad\bigskip & b_{(F^{\ast},G^{\ast})}=-%
\overline{B_{(F,G)}} \\ 
A_{(F^{\ast},G^{\ast})}=-A_{(F,G)} & B_{(F^{\ast},G^{\ast})}=-\overline {%
b_{(F,G)}}%
\end{array}%
\end{equation*}
\end{proposition}

\begin{proposition}
\label{adjointsucessor}If $\left( F_{-1},G_{-1}\right) $ is a predecessor of 
$(F,G)$ then $(F_{-1},G_{-1})^{\ast}$ is a successor of $(F,G)^{\ast}$.
\end{proposition}

\begin{definition}
Let $\Gamma$ be a rectifiable curve leading from $z_{0}$ to $z_{1}$ and $W$
a continuous function on $\Gamma$. The $(F,G)$-*- integral of $W$ along $%
\Gamma$ is defined by%
\begin{equation*}
\ast\int_{\Gamma}Wd_{(F,G)}z=\limfunc{Sc}\int_{\Gamma}G^{\ast }Wdz+j\limfunc{%
Sc}\int_{\Gamma}F^{\ast}Wdz
\end{equation*}
and the $(F,G)$-integral is defined as%
\begin{equation*}
\int_{\Gamma}Wd_{(F,G)}z=F(z_{1})\limfunc{Sc}\int_{\Gamma}G^{\ast
}Wdz+G(z_{1})\limfunc{Sc}\int_{\Gamma}F^{\ast}Wdz.
\end{equation*}
\end{definition}

\begin{definition}
A continuous $\mathbb{B}$-valued function $W$ defined in $\Omega$ is called $%
(F,G)$-integrable if for every closed curve $\Gamma$ lying in a simply
connected subdomain of $\Omega$,%
\begin{equation*}
\ast\int_{\Gamma}Wd_{(F,G)}z=0\text{.}
\end{equation*}
\end{definition}

\begin{proposition}
\label{propMorera}Let $(F,G)$ be a predecessor of $(F_{1},G_{1})$ and $W$ be
a continuous function defined in a simply connected domain $\Omega$. Then

\begin{description}
\item[$(i)$] $W$ is $(F_{1,}G_{1})$-pseudoanalytic in $\Omega$ iff $W$ is $%
(F,G)$-integrable, that is iff%
\begin{equation*}
\limfunc{Sc}\int_{\Gamma}G^{\ast}Wdz=\limfunc{Sc}\int_{\Gamma }F^{\ast}Wdz=0
\end{equation*}
along every closed path $\Gamma\subset\Omega$.

\item[$(ii)$] If $W$ is $(F_{1,}G_{1})$-pseudoanalytic in $\Omega$ and $%
z_{0}\in\Omega$ then $w:=\int_{z_{0}}^{z}W(\tau)d_{(F,G)}\tau$ is $(F,G)$%
-pseudoanalytic and $W=\dfrac{d_{(F,G)}w}{dz}$.
\end{description}
\end{proposition}

\begin{remark}
The above proposition remains true for multiply-connected domains, except
that in $(ii)$ $w$ may be multiple valued.
\end{remark}

\subsection{Formal powers\label{Subsect Formal powers}}

Let $(F,G)$ be a generating pair corresponding to (\ref{Vekua_equation (F,G)}%
) in some domain $\Omega\subset\mathbb{R}^{2}$, $n\in\mathbb{Z}$ and $%
z_{0}=x_{0}+jy_{0}\in\Omega$. We call formal powers of order $n$ and center $%
z_{0}$ to a pair of solutions $Z^{(n)}(1,z_{0},z)$, $Z^{(n)}(j,z_{0},z)$ of (%
\ref{Vekua_equation (F,G)}) in $\Omega\backslash\left\{ z_{0}\right\} $ such
that 
\begin{equation}
Z^{(n)}(1,z_{0},z)\sim\left( z-z_{0}\right) ^{n},\quad
Z^{(n)}(j,z_{0},z)\sim j\left( z-z_{0}\right) ^{n},\quad\text{as }%
z\rightarrow z_{0}\text{,}  \label{FP-assimp}
\end{equation}
where $\left( z-z_{0}\right) ^{n}=\left[ (x-x_{0})+j(y-y_{0})\right] ^{n}$
are the usual powers of the $\mathbb{B}$-analytic functions. For $\alpha \in%
\mathbb{B}$, the following definition will be useful%
\begin{equation}
Z^{(n)}(\alpha,z_{0},z):=\limfunc{Sc}\alpha\text{ }Z^{(n)}(1,z_{0},z)+%
\limfunc{Vec}\alpha\text{ }Z^{(n)}(j,z_{0},z).  \label{FP_coefalpha}
\end{equation}
The function defined by (\ref{FP_coefalpha}) is called formal power of the
order $n$ with the coefficient $\alpha$ and the center $z_{0}$. In this work
we are mainly interested in negative formal powers, that is when $n<0$. We
emphasize that the formal powers are not uniquely defined. For example, if a
regular solution of (\ref{Vekua_equation (F,G)}) is added to a negative
formal power the resulting solution will be again a negative formal power of
the same order, center and coefficient as the initial one.

The special case when $n=-1$ is distinguished: the formal power $%
Z^{(-1)}(\alpha,z_{0},z)$ playing an important role in the study of
pseudoanalytic functions is called Cauchy kernel. Next we obtain some
asymptotic formulas for the Cauchy kernel that will be important in order to
establish the Cauchy integral formula in the subsequent section. First we
introduce the following definition.

\begin{definition}
Let $W$, $g$ be a $\mathbb{B}$-valued and an $\mathbb{R}$-valued functions
respectively. We agree that the notation $W=\mathcal{O}(g),$ as $%
z\rightarrow z_{0},$ means that the function $\dfrac{W(z)}{g(z)}$ is bounded
in some neighborhood of $z_{0}$.
\end{definition}

\begin{proposition}
\label{C_K_assimpt}For $\alpha=1$or $\alpha=j$ the asymptotic formulas%
\begin{equation}
\lim_{z\rightarrow z_{0}}\left\vert \frac{\overline{Z^{(-1)}(\alpha,z_{0},z)}%
}{Z^{(-1)}(\alpha,z_{0},z)}\right\vert =1  \label{c-k-bounded}
\end{equation}
and 
\begin{equation}
Z^{(-1)}(\alpha,z_{0},z)=\frac{\alpha}{z-z_{0}}+\mathcal{O(}\log\left\vert
z-z_{0}\right\vert ),\quad\text{as \ }z\rightarrow z_{0}
\label{C-K-difference}
\end{equation}
hold.
\end{proposition}

\begin{proof}
We prove the theorem for $\alpha=1$, the case when $\alpha=j$ can be treated
by analogy.

$(i)$ By the definition of $Z^{(-1)}(1,z_{0},z)$ we obtain 
\begin{equation}
\lim_{z\rightarrow z_{0}}(z-z_{0})Z^{(-1)}(1,z_{0},z)=1.
\label{def_kernelcoef1_prop}
\end{equation}
From the fact that $\mathcal{R}(\mathbb{B)}$ is open (Proposition \ref%
{R(B)_open_prop}) and $1\in\mathcal{R}(\mathbb{B)}$ we conclude that there
exists some neighborhood of $z_{0}$ denoted by $N_{z_{0}}$ in which $%
(z-z_{0})Z^{(-1)}(1,z_{0},z)\in\mathcal{R}(\mathbb{B)}$. Particularly $%
Z^{(-1)}(1,z_{0},z)\in\mathcal{R}(\mathbb{B)}$ in this neighborhood because $%
(z-z_{0})\in\mathcal{R}(\mathbb{B)}$ for all $z$, $z_{0}\in\mathbb{C}_{j}$.
Hence the function on the left hand side of (\ref{c-k-bounded}) is well
defined for $z\in N_{z_{0}}$. Finally, using (\ref{def_kernelcoef1_prop})
and Proposition \ref{Prop_norm} we have%
\begin{align*}
\left\vert \frac{\overline{Z^{(-1)}(1,z_{0},z)}}{Z^{(-1)}(1,z_{0},z)}%
\right\vert & =\left\vert \frac{\overline{(z-z_{0})}}{(z-z_{0})}\right\vert
\left\vert \frac{\overline{Z^{(-1)}(1,z_{0},z)}}{Z^{(-1)}(1,z_{0},z)}%
\right\vert = \\
& =\left\vert \frac{\overline{(z-z_{0})}}{(z-z_{0})}\frac{\overline {%
Z^{(-1)}(1,z_{0},z)}}{Z^{(-1)}(1,z_{0},z)}\right\vert \rightarrow1
\end{align*}
when $z\rightarrow z_{0}$.

$(ii)$ Let $N_{z_{0}}$ be the neighborhood of $z_{0}$ from $(i)$ in which
the left side of (\ref{c-k-bounded}) is bounded. Theorem \ref{Similarity. P}
guarantees the existence of a $\mathbb{B}$-analytic function $\Psi$ in $%
N_{z_{0}}\backslash\left\{ z_{0}\right\} $ such that%
\begin{equation*}
Z^{(-1)}(1,z_{0},z)=\Psi(z)E\left[ S(z)\right] ,
\end{equation*}
where $S=T_{\overline{z}}\left[ a(z)+\frac{\overline{Z^{(-1)}(1,z_{0},z)}}{%
Z^{(-1)}(1,z_{0},z)}b(z)\right] $. We see that $\underset{z\rightarrow z_{0}}%
{\lim}(z-z_{0})\Psi(z)=E\left[ -S(z_{0})\right] $ which allows one to
conclude that 
\begin{equation*}
E\left[ S(z_{0})\right] \Psi(z)=\frac{1}{z-z_{0}}+R\left( z\right)
\end{equation*}
where $R\left( z\right) $ is a $\mathbb{B}$-analytic function in $N_{z_{0}}$%
. Then we have%
\begin{equation}
\begin{array}{c}
Z^{(-1)}(1,z_{0},z)-\dfrac{1}{z-z_{0}}= \\ 
=\dfrac{1}{z-z_{0}}\left( E\left[ S(z)-S(z_{0})\right] -1\right) +R\left(
z\right) E\left[ S(z)-S(z_{0})\right] .%
\end{array}
\label{kernels-dif}
\end{equation}
From Propositions \ref{Ez_prop} and \ref{Tz_prop} (in both part $(ii)$) we
obtain the following inequalities in $N_{z_{0}}$%
\begin{align*}
\left\vert E\left[ S(z)-S(z_{0})\right] -1\right\vert & \leq e^{2\left\vert
S(z)-S(z_{0})\right\vert }-1\leq M_{1}\left\vert S(z)-S(z_{0})\right\vert \\
& \leq M_{1}M_{2}\left\vert z-z_{0}\right\vert \left\{ 1+\log^{+}\frac {1}{%
\left\vert z-z_{0}\right\vert }\right\}
\end{align*}
which together with (\ref{kernels-dif}) implies the assertion.
\end{proof}

\bigskip

As we mentioned above if no other condition is imposed then the formal
powers are not uniquely defined. However, as we discuss next, under certain
conditions regarding the point of infinity they are uniquely determined.
Until the end of this subsection we suppose that $F$ and $G$ are $\mathbb{C}%
_{j}$-valued functions defined everywhere and that $(F,G)$ is a complete
generating pair, that is, $(F,G)$ is a generating pair in $\mathbb{R}^{2}$
such that $F(\infty)$, $G(\infty)$ exist, $Vec(\overline{F(\infty)}%
G(\infty)\neq0$ and the functions $F(\frac{1}{z})$, $G(\frac{1}{z})$ are H%
\"{o}lder continuous.

The results of L. Bers \cite{BersFP} together with Remark \ref%
{remark_decouple_vekua} allow us to obtain the following statements.

\begin{proposition}
\label{prop_global_FP}Under the above conditions for any integer $n$ and for
every pair of numbers $\alpha\in\mathbb{B}$, $\alpha\neq0$ and $z_{0}\in%
\mathbb{C}_{j}$ there exists one and only one $(F,G)$-pseudoanalytic
function $Z^{(n)}(\alpha,z_{0},z)$ defined in $\mathbb{C}_{j}\backslash
\left\{ z_{0}\right\} $ such that%
\begin{equation*}
Z^{(n)}(\alpha,z_{0},z)\sim\alpha\left( z-z_{0}\right) ^{n},\quad
z\rightarrow z_{0}
\end{equation*}
and%
\begin{equation}
Z^{(n)}(\alpha,z_{0},z)=\mathcal{O}(\left\vert z\right\vert ^{n}),\quad
z\rightarrow\infty  \label{CK_cond_infinity}
\end{equation}
and having no other zeros or poles except at $z=\infty$.
\end{proposition}

The functions described in the above proposition are called global formal
powers. The negative global formal powers enjoy the following interesting
relations.

\begin{theorem}
\label{theoremFPRelation}Let $\left\{ (F_{m},G_{m})\right\} $, $m=0,\pm
1,\pm 2,\ldots $ be a generating sequence of complete generating pairs in
which $(F,G)$ is embedded and $Z_{m}^{(-n)}(\alpha ,z_{0},z)$, $Z_{m}^{\ast
(-n)}(\alpha ,z_{0},z)$ denote the $(F_{m},G_{m})$- and the $%
(F_{m},G_{m})^{\ast }$- negative formal powers, respectively. Then for any
positive integer $n$ we have%
\begin{align*}
\limfunc{Sc}Z_{m}^{(-n)}(j,z_{0},z)+j\limfunc{Sc}Z_{m}^{(-n)}(1,z_{0},z)&
=(-1)^{n}Z_{m-n}^{\ast (-n)}(j,z,z_{0})\bigskip \\
\limfunc{Vec}Z_{m}^{(-n)}(j,z_{0},z)+j\limfunc{Vec}Z_{m}^{(-n)}(1,z_{0},z)&
=(-1)^{n}Z_{m-n}^{\ast (-n)}(1,z,z_{0}).
\end{align*}
\end{theorem}

Due to the above theorem for fixed values of $z$ the global formal powers $%
Z_{m}^{(-n)}(j,z_{0},z)$ are continuously differentiable in the variable $%
z_{0}$.

There are currently no results regarding the existence or the uniqueness for
the global formal powers of a general bicomplex Vekua equation though in
Subsection \ref{Subsect First Cauchy} we establish an existence result under
certain additional conditions. We end this section by mentioning that even
for complex Vekua equations the construction of the negative formal powers
(global or not) is a very difficult task. In Section \ref{Sect Applications}
we construct the negative formal powers in the explicit form for some
classes of bicomplex main Vekua equations.

\subsection{A relation between the Schr\"{o}dinger equation and the main
Vekua equation}

Let $q$ be a complex ($C_{i}$-valued) continuous function in $\Omega
\subseteq\mathbb{R}^{2}$. Consider the two-dimensional stationary Schr\"{o}%
dinger equation%
\begin{equation}
\Delta u-qu=0\quad\text{in }\Omega  \label{sch}
\end{equation}
and assume that it possesses a $C_{i}$-valued particular solution $f\in
C^{2}(\Omega)$ such that $f(z)\neq0$ for all $z\in\Omega$. We will need a
result from \cite{Krpseudoan} (see also \cite{APFT}) relating the Schr\"{o}%
dinger equation with a Vekua equation of a special kind.

\begin{theorem}
\label{schandmainvek}Let $W=u+jv$ with $u=\limfunc{Sc}W$, $v=\limfunc{Vec}W$
be a solution of the main bicomplex Vekua equation%
\begin{equation}
\partial_{\overline{z}}W=\frac{\partial_{\overline{z}}f}{f}\overline{W}\quad%
\text{in }\Omega.  \label{mainvekua}
\end{equation}
Then $u$ is a solution of (\ref{sch}) and $v$ is a solution of%
\begin{equation}
\Delta v-q_{1}v=0\quad\text{in }\Omega  \label{sch1}
\end{equation}
where $q_{1}=2\frac{\left( \partial_{x}f\right) ^{2}+\left( \partial
_{y}f\right) ^{2}}{f^{2}}-q.$
\end{theorem}

\begin{remark}
\label{Remark_coeficients _main_vekua}A generating pair corresponding to (%
\ref{mainvekua}) can be chosen in the form 
\begin{equation}
(F,G)=(f,\frac{j}{f}).  \label{Gen pair for the main Vekua}
\end{equation}
The corresponding characteristic coefficients are%
\begin{equation*}
a_{(f,\frac{j}{f})}=A_{(f,\frac{j}{f})}=0,\quad b_{(f,\frac{j}{f})}=\frac{%
\partial_{\overline{z}}f}{f},\quad B_{(f,\frac{j}{f})}=\frac {\partial_{z}f}{%
f}
\end{equation*}
and the successor equation of (\ref{mainvekua}) is%
\begin{equation}
\partial_{\overline{z}}W=-\frac{\partial_{z}f}{f}\overline{W}.
\label{Main_vekua_suc}
\end{equation}
\end{remark}

We need the following notation. Let $W$ be a $\mathbb{B}$-valued function
defined on a simply connected domain $\Omega$ with $u=\func{Sc}W$ and $v=%
\func{Vec}W$ such that 
\begin{equation}
\frac{\partial u}{\partial y}-\frac{\partial v}{\partial x}=0,\quad\text{in }%
\Omega  \label{compcond}
\end{equation}
and let $\Gamma\subset\Omega$ be a simple rectifiable curve leading from $%
z_{0}=x_{0}+jy_{0}$ to $z=x+jy$. Then the integral 
\begin{equation*}
\overline{A}W(z):=2\left( \int_{\Gamma}udx+vdy\right)
\end{equation*}
is path-independent, and all $\mathbb{C}_{i}$-valued solutions $\varphi$ of
the equation $\partial\overline{_{z}}\varphi=W$ in $\Omega$ have the form $%
\varphi(z)=\overline{A}W(z)+c$ where $c$ is an arbitrary $\mathbb{C}_{i}$%
-constant.

\begin{remark}
\label{T_f_def}It is easy to check that if $u$ is a solution of the Schr\"{o}%
dinger equation (\ref{sch}) in a simply connected domain $\Omega$ then the $%
\mathbb{B}$-valued function $jf^{2}\partial_{\overline{z}}(f^{-1}u)$
satisfies (\ref{compcond}) and therefore the function constructed by the rule%
\begin{equation}
T_{f}(u):=f^{-1}\overline{A}(jf^{2}\partial_{\overline{z}}(f^{-1}u))
\label{Darbouxttransf}
\end{equation}
is well defined in $\Omega$. Application of $T_{f}$ to a more general class
of functions will be also considered but as in this case the result of such
operation may depend on the choice of the curve $\Gamma$ leading from $z_{0}$
to $z$, the notation $T_{f,\Gamma}(u)$ will be used.
\end{remark}

\begin{theorem}
\label{Theo_Tf}\cite{KrJPhys06} If $u$ is a $\mathbb{C}_{i}$-valued solution
of the Schr\"{o}dinger equation (\ref{sch}) in a simply connected domain $%
\Omega$ then $v=T_{f}(u)$ is a solution of (\ref{sch1}) and $W=u+jv$ is a
solution of (\ref{mainvekua}).
\end{theorem}

\begin{corollary}
\label{conv_sch}Let $u_{n}\in C^{2}(\Omega)$ be a sequence of solutions of (%
\ref{sch}) and $u$ be such that $u_{n}\rightarrow u$, $\partial_{x}u_{n}%
\rightarrow\partial_{x}u$ and $\partial_{y}u_{n}\rightarrow\partial_{y}u$
uniformly in $\Omega$. Then $u$ is a solution of (\ref{sch})$.$
\end{corollary}

\begin{proof}
The assumption together with the above theorem allows one to conclude that
the sequence of solutions of (\ref{mainvekua}) given by $%
W_{n}:=u_{n}+jT_{f}(u_{n})$ converges uniformly to $W:=u+jT_{f}(u)$. Then by
Theorem \ref{theorem_pseudo_conv}, $W$ is a solution of (\ref{mainvekua})
and from Theorem \ref{schandmainvek} its scalar part $u$ is a solution of (%
\ref{sch}).
\end{proof}

\section{Cauchy integral formulas for bicomplex pseudoanalytic functions 
\label{Sect Cauchy Int Formulas}}

Consider the bicomplex Vekua equation%
\begin{equation}
\partial _{\overline{z}}W=aW+b\overline{W}  \label{vekua}
\end{equation}%
in a simply connected domain $\Sigma \subset \mathbb{C}_{j}$ where the $%
\mathbb{B}$-valued functions $a$ and $b$ are H\"{o}lder continuous.
Throughout this section $\Omega $ is a bounded domain such that $\overline{%
\Omega }\subset \Sigma $ and regular in the following sense \cite[p. 2]{Bers
Book} $\partial \Omega $ consists of a finite number of piecewise
differentiable simple closed Jordan curves.

\subsection{First Cauchy's integral formula\label{Subsect First Cauchy}}

For complex pseudoanalytic functions the integral representation called here
first Cauchy's integral formula was obtained in \cite{Vekua}. In that
representation the Cauchy kernel corresponding to an adjoint Vekua equation
is used. In the present section we obtain this Cauchy integral formula for $%
\mathbb{B}$-pseudoanalytic functions.

The adjoint equation of (\ref{vekua}) has the form%
\begin{equation}
\partial_{\overline{z}}V=-aV-\overline{b}\overline{V}.  \label{vekuaadjoint}
\end{equation}
In the next two propositions a useful characterization of solutions of (\ref%
{vekuaadjoint}) is obtained.

\begin{proposition}
\label{cauchyintegraltheorem}Let $W$ and $V$ be solutions of (\ref{vekua})
and (\ref{vekuaadjoint}) respectively in $\Omega$ and continuous in $%
\overline {\Omega}$. Then%
\begin{equation}
\limfunc{Vec}\int_{\partial\Omega}W(\tau)V(\tau)d\tau=0.
\label{intcauchyformula}
\end{equation}
\end{proposition}

\begin{proof}
The proof is analogous to the proof given in the complex situation \cite[p.
169]{Vekua}.
\end{proof}

The converse of this statement can be obtained implementing Bers' result on
a generalization of Morera's theorem for solutions of equation (\ref%
{vekuaadjoint}).

\bigskip

\begin{proposition}
\label{propmoreraadjunta}Let $(F,G)$ be a generating pair corresponding to (%
\ref{vekua}) and $V$ be a continuous function in $\Omega$. Then $V$ is a
solution of (\ref{vekuaadjoint}) in $\Omega$ iff the following equalities
hold%
\begin{equation}
\limfunc{Vec}\int_{\Gamma}G(\tau)V(\tau)d\tau=\limfunc{Vec}%
\int_{\Gamma}F(\tau)V(\tau)d\tau=0  \label{moreracondition}
\end{equation}
along every closed path $\Gamma$ situated in a simply connected subdomain of 
$\Omega$.
\end{proposition}

\begin{proof}
If $V$ is a solution of (\ref{vekuaadjoint}) then (\ref{moreracondition}) is
a consequence of the previous proposition. To prove the statement in the
opposite direction notice that if $(F,G)$ is a generating pair corresponding
to (\ref{vekua}) then from Proposition \ref{cara_coef_(F,G)*} one has that $%
(jF^{\ast },jG^{\ast })$ is a generating pair corresponding to the equation%
\begin{equation*}
W_{\overline{z}}=-aW+\overline{B_{(F,G)}}\overline{W},
\end{equation*}%
and its successor is equation (\ref{vekuaadjoint}). This, together with
Proposition \ref{propMorera} implies that the solutions of (\ref%
{vekuaadjoint}) are $(jF^{\ast },jG^{\ast })$-integrable functions. Notice
that $(jF^{\ast },jG^{\ast })^{\ast }=(-jF,-jG)$ and then equality (\ref%
{moreracondition}) tells us that $V$ is $(jF^{\ast },jG^{\ast })$%
-integrable. Thus, $V$ is a solution of (\ref{vekuaadjoint}).
\end{proof}

\begin{proposition}[First Cauchy's integral formula]
Let $W$ be a solution of (\ref{vekua}) in $\Omega$ continuous in $\overline{%
\Omega}$ and $\widehat{Z}^{(-1)}(\alpha,z_{0},z)$ be a Cauchy kernel of (\ref%
{vekuaadjoint}) in $\Sigma$. Then,%
\begin{equation*}
\limfunc{Vec}\int_{\partial\Omega}W(\tau)\widehat{Z}^{(-1)}(1,z_{0},\tau)d%
\tau-j\limfunc{Vec}\int_{\partial\Omega}W(\tau)\widehat{Z}%
^{(-1)}(j,z_{0},\tau)d\tau
\end{equation*}%
\begin{equation*}
=\left\{ 
\begin{array}{c}
2\pi W(z_{0}), \\ 
0,%
\end{array}%
\begin{array}{c}
\text{if }z_{0}\in\Omega \\ 
\text{if }z_{0}\in\Sigma\backslash\left\{ \overline{\Omega}\right\}%
\end{array}
\right. .
\end{equation*}
\end{proposition}

\begin{proof}
When $z_{0}\in \Sigma \backslash \left\{ \overline{\Omega }\right\} $ the
equality follows from Proposition \ref{cauchyintegraltheorem}. For $z_{0}\in
\Omega $ let us prove the scalar part of the equality, 
\begin{equation*}
\limfunc{Vec}\int_{\partial \Omega }W(\tau )\widehat{Z}^{(-1)}(1,z_{0},\tau
)d\tau =2\pi \limfunc{Sc}W(z_{0}).
\end{equation*}%
Let $D_{\varepsilon }(z_{0})$ be a disk with the center $z_{0}$ and radius $%
\varepsilon $ and $\Omega _{\varepsilon }=\Omega \backslash \left\{
D_{\varepsilon }(z_{0})\right\} $. Then for a sufficiently small $%
\varepsilon $ from Proposition \ref{cauchyintegraltheorem} we obtain%
\begin{equation*}
\limfunc{Vec}\int_{\partial \Omega _{\varepsilon }}W(\tau )\widehat{Z}%
^{(-1)}(1,z_{0},\tau )d\tau =0,
\end{equation*}%
that is%
\begin{equation*}
\limfunc{Vec}\int_{\partial \Omega }W(\tau )\widehat{Z}^{(-1)}(1,z_{0},\tau
)d\tau =\limfunc{Vec}\int_{\partial D_{\varepsilon }}W(\tau )\widehat{Z}%
^{(-1)}(1,z_{0},\tau )d\tau .
\end{equation*}%
By (\ref{C-K-difference}) we have that%
\begin{equation*}
\underset{\varepsilon \rightarrow 0}{\lim }\int_{\partial D_{\varepsilon
}}W(\tau )\widehat{Z}^{(-1)}(1,z_{0},\tau )d\tau =\underset{\varepsilon
\rightarrow 0}{\lim }\int_{\partial D_{\varepsilon }}\frac{W(\tau )}{\tau
-z_{0}}d\tau =2\pi jW(z_{0}).
\end{equation*}%
Thus,%
\begin{equation*}
\limfunc{Vec}\int_{\partial \Omega }W(\tau )\widehat{Z}^{(-1)}(1,z_{0},\tau
)d\tau =\limfunc{Vec}(2\pi jW(z_{0}))=2\pi \limfunc{Sc}W(z_{0}).
\end{equation*}%
The vector part of the Cauchy integral formula is proved in the same way.
\end{proof}

\begin{remark}
As the adjoint of equation of (\ref{vekuaadjoint}) is (\ref{vekua}) a
similar Cauchy's integral formula holds for solutions of (\ref{vekuaadjoint}%
) in terms of Cauchy kernels corresponding to (\ref{vekua}).
\end{remark}

Following \cite[p. 174]{Vekua} we use the first Cauchy integral formula to
obtain relations between Cauchy kernels corresponding to mutually adjoint
bicomplex Vekua equations.

\begin{proposition}
\label{propkernels}Let $\Sigma=\mathbb{C}_{j}$ and suppose that $%
Z^{(-1)}(\alpha,z_{0},z)$, $\widehat{Z}^{(-1)}(\alpha,z_{0},z)$ are Cauchy
kernels in $\mathbb{C}_{j}$ corresponding to (\ref{vekua}) and (\ref%
{vekuaadjoint}) respectively, both behaving like $\mathcal{O}(\left\vert
z\right\vert ^{-1})$ as $z\rightarrow\infty$. Then%
\begin{equation}
\widehat{Z}^{(-1)}(1,z_{0},z)=-\limfunc{Sc}Z^{(-1)}(1,z,z_{0})+j\limfunc{Sc}%
Z^{(-1)}(j,z,z_{0})  \label{rela1}
\end{equation}
and%
\begin{equation}
\widehat{Z}^{(-1)}(j,z_{0},z)=\limfunc{Vec}Z^{(-1)}(1,z,z_{0})-j\limfunc{Vec}%
Z^{(-1)}(j,z,z_{0}).  \label{rela2}
\end{equation}
\end{proposition}

\begin{proof}
Let $z_{0},z\in \mathbb{C}_{j}$, $z_{0}\neq z$ and $\varepsilon >0$ be
sufficiently small such that $z\in \Omega _{\varepsilon }=D_{\frac{1}{%
\varepsilon }}(z_{0})\backslash D_{\varepsilon }(z_{0})$ where $%
D_{\varepsilon }(z_{0})$ is a disc with the center $z_{0}$ and radius $%
\varepsilon $. Then from the first Cauchy integral formula we have%
\begin{align*}
& 2\pi \widehat{Z}^{(-1)}(1,z_{0},z) \\
& =\limfunc{Vec}\int_{\partial \Omega _{\varepsilon }}\widehat{Z}%
^{(-1)}(1,z_{0},\tau )Z^{(-1)}(1,z,\tau )d\tau -j\limfunc{Vec}\int_{\partial
\Omega \varepsilon }\widehat{Z}^{(-1)}(1,z_{0},\tau )Z^{(-1)}(j,z,\tau )d\tau
\end{align*}%
or equivalently,%
\begin{align*}
& 2\pi \widehat{Z}^{(-1)}(1,z_{0},z) \\
& =\limfunc{Vec}\int_{\partial D_{\frac{1}{\varepsilon }}}\widehat{Z}%
^{(-1)}(1,z_{0},\tau )Z^{(-1)}(1,z,\tau )d\tau -j\limfunc{Vec}\int_{\partial
D_{\frac{1}{\varepsilon }}}\widehat{Z}^{(-1)}(1,z_{0},\tau
)Z^{(-1)}(j,z,\tau )d\tau \\
& -\limfunc{Vec}\int_{\partial D_{\varepsilon }}\widehat{Z}%
^{(-1)}(1,z_{0},\tau )Z^{(-1)}(1,z,\tau )d\tau +j\limfunc{Vec}\int_{\partial
D_{\varepsilon }}\widehat{Z}^{(-1)}(1,z_{0},\tau )Z^{(-1)}(j,z,\tau )d\tau .
\end{align*}%
Taking into account the asymptotic behaviour of the kernels at infinity it
is easy to verify that the first and the second integrals on the right-hand
side of the equality tends to zero when $\varepsilon \rightarrow 0$. Then (%
\ref{rela1}) is obtained from the last equality applying the first Cauchy
integral formula to the last two integrals and considering $\varepsilon
\rightarrow 0$. A similar reasoning proves (\ref{rela2}).
\end{proof}

\begin{remark}
Under the assumptions of Proposition \ref{prop_global_FP} (which in
particular imply that $a$ and $b$ are $C_{j}$-valued functions) equations (%
\ref{vekua}) and (\ref{vekuaadjoint}) possess unique Cauchy kernels (called
global Cauchy kernels) satisfying the conditions of Proposition \ref%
{propkernels} and therefore they can be computed one from another by means
of (\ref{rela1}) and (\ref{rela2}). For a general bicomplex Vekua equation
we were not able to obtain comparable definitive results regarding the
existence or uniqueness of the kernels possessing such properties. However
from the previous proposition we immediately obtain the following uniqueness
result.
\end{remark}

\begin{corollary}
If both equations (\ref{vekua}) and (\ref{vekuaadjoint}) possess Cauchy
kernels fulfilling the assumptions of Proposition \ref{propkernels} then
they are unique.
\end{corollary}

\subsection{Second Cauchy's integral formula and reproducing Cauchy kernels}

\begin{definition}
Let $\Omega\subseteq\Sigma$ be a regular domain, $W$ be a continuous
function up to $\overline{\Omega}$ and $Z^{(n)}(\alpha,z_{0},z)$ be a formal
power corresponding to (\ref{vekua}) in $\Sigma$, continuous in the variable 
$z_{0}$ for fixed values of $z$. Following \cite{Bers Book} we define%
\begin{equation*}
\int_{\partial\Omega}Z^{(n)}(jW(\tau)d\tau,\tau,z_{0})=%
\int_{a}^{b}Z^{(n)}(jW(\gamma(t))\gamma^{\prime}(t)dt,\gamma(t),z_{0})
\end{equation*}
where $\gamma(t):\left[ a,b\right] \rightarrow\mathbb{C}_{j}$ is a
parametrization of $\partial\Omega$.
\end{definition}

\begin{theorem}[Second Cauchy's integral formula]
\cite{Vekua}\label{SecondCIF} Let $Z^{(-1)}(\alpha ,z_{0},z)$, $\widehat{Z}%
^{(-1)}(\alpha ,z_{0},z)$ be Cauchy kernels corresponding to (\ref{vekua})
and (\ref{vekuaadjoint}) respectively satisfying (\ref{rela1}), (\ref{rela2}%
) in $\Sigma $ and $W$ be a solution of (\ref{vekua}) in $\Omega $
continuous in $\overline{\Omega }$. Then%
\begin{equation}
\int_{\partial \Omega }Z^{(-1)}(jW(\tau )d\tau ,\tau ,z_{0})=2\pi W(z_{0}),
\label{reproductor}
\end{equation}%
for $z_{0}\in \Omega $ and%
\begin{equation}
\int_{\partial \Omega }Z^{(-1)}(jW(\tau )d\tau ,\tau ,z_{0})=0
\label{reproductor1}
\end{equation}%
for $z_{0}\in \Sigma \backslash \left\{ \overline{\Omega }\right\} $.
\end{theorem}

\begin{proof}
Direct calculation gives us the equality%
\begin{equation*}
\int_{\partial\Omega}Z^{(-1)}(jW(\tau)d\tau,\tau,z_{0})=
\end{equation*}%
\begin{align}
& =\limfunc{Vec}\int_{\partial\Omega}W(\tau)\left\{ -\limfunc{Sc}%
Z^{(-1)}(1,\tau,z_{0})+j\limfunc{Sc}Z^{(-1)}(j,\tau,z_{0})\right\} d\tau
\label{rela3} \\
& -j\limfunc{Vec}\int_{\partial\Omega}W(\tau)\left\{ \limfunc{Vec}%
Z^{(-1)}(1,\tau,z_{0})-j\limfunc{Vec}Z^{(-1)}(j,\tau,z_{0})\right\} d\tau. 
\notag
\end{align}

From (\ref{rela1}), (\ref{rela2}), (\ref{rela3}) and from the first Cauchy
integral formula we obtain the result.
\end{proof}

\begin{remark}
\label{RemExample}Formulas (\ref{reproductor}) and (\ref{reproductor1}) were
obtained independently in \cite{Bers Book} and \cite{Vekua} using global
Cauchy kernels. Our proof which follows that from \cite{Vekua} reveals that
the second Cauchy integral formula is equivalent to the first when an
appropriate Cauchy kernel corresponding to (\ref{vekua}) is used. The Cauchy
kernels involved, $Z^{(-1)}(\alpha ,z_{0},z)$ and $\widehat{Z}^{(-1)}(\alpha
,z_{0},z)$ should satisfy equalities (\ref{rela1}), (\ref{rela2}) in the
domain of interest but not necessarily be global. It is worth noticing that
equalities (\ref{reproductor}) and (\ref{reproductor1}) are not valid for an
arbitrary Cauchy kernel of (\ref{vekua}). For example, consider the
following Cauchy kernels%
\begin{equation*}
Z^{(-1)}(1,\zeta ,z)=\frac{1}{z-\zeta }+\xi ,\quad Z^{(-1)}(j,\zeta ,z)=%
\frac{j}{z-\zeta }.
\end{equation*}%
of the equation $\partial _{\overline{z}}W=0$, where $z=x+jy$ and $\zeta
=\xi +j\eta $. If $\Omega $ is the unitary disk, $z_{0}=0$ and $W\equiv 1$
then the left-hand side of (\ref{reproductor}) is%
\begin{align*}
\dint\limits_{0}^{2\pi }Z^{(-1)}(j^{2}e^{j\theta }d\theta ,e^{j\theta },0)&
=-\dint\limits_{0}^{2\pi }\left[ \left( \cos \theta \right)
Z^{(-1)}(1,e^{j\theta },0)+\left( \sin \theta \right) Z^{(-1)}(j,e^{j\theta
},0)\right] d\theta \\
& =\dint\limits_{0}^{2\pi }(1-\cos ^{2}\theta )d\theta =\pi
\end{align*}%
which is different from $2\pi W(0).$
\end{remark}

\begin{definition}
We say that a Cauchy kernel $Z^{(-1)}(\alpha,z_{0},z)$ defined in $\Sigma$
reproduces a solution $W$ of (\ref{vekua}) if for every domain $\Omega$
where $W$ is pseudoanalytic in $\Omega$ and continuous in $\overline{\Omega}$
the formula (\ref{reproductor}) holds. A local Cauchy kernel is called
reproducing if it reproduces all solutions of (\ref{vekua}).
\end{definition}

Combining results of the preceding sections we obtain the following criteria
describing the reproducing Cauchy kernels.

\begin{theorem}
\label{Teo_rep_ker}Let $(F,G)$ be a generating pair and $Z^{(-1)}(\alpha
,z_{0},z)$ be a Cauchy kernel, both corresponding to (\ref{vekua}) in $%
\Sigma $ and such that for each fixed $z$ the kernel $Z^{(-1)}(%
\alpha,z_{0},z)$ is a continuous function in the variable $z_{0}$ and%
\begin{equation}
\lim_{z_{0}\rightarrow z}(z-z_{0})Z^{(-1)}(\alpha,z_{0},z)=\alpha\text{%
,\quad }\alpha=1,j.  \label{cauchykernelstrong}
\end{equation}
Then the following statements are equivalent:

\begin{description}
\item[$(i)$] $Z^{(-1)}(\alpha,z_{0},z)$ reproduces both solutions $F$ and $G 
$;

\item[$(ii)$] There exists a Cauchy kernel $\widehat{Z}^{(-1)}(\alpha
,z_{0},z)$ corresponding to (\ref{vekuaadjoint}) such that (\ref{rela1}) and
(\ref{rela2}) hold;

\item[$(iii)$] Second Cauchy's integral formula holds;

\item[$(iv)$] $Z^{(-1)}(\alpha,z_{0},z)$ is a reproducing Cauchy kernel.
\end{description}
\end{theorem}

\begin{proof}

\begin{description}
\item[$(i)\Rightarrow(ii)$] Let $z_{0}\in\Sigma$ and $\Omega$ be a domain
such that $z_{0}\notin\overline{\Omega}\subset\Sigma$. Then from $(i)$ we
have%
\begin{equation*}
\int_{\partial\Omega}Z^{(-1)}(jF(\tau)d\tau,\tau,z_{0})=\int_{\partial\Omega
}Z^{(-1)}(jG(\tau)d\tau,\tau,z_{0})=0.
\end{equation*}
From this equality and from (\ref{rela3}) we obtain 
\begin{equation}
\limfunc{Vec}\int_{\partial\Omega}F(\tau)w_{1,2}(z_{0},\tau )d\tau=\limfunc{%
Vec}\int_{\partial\Omega}G(\tau)w_{1,2}(z_{0},\tau )d\tau=0  \label{rela4}
\end{equation}
where%
\begin{align*}
w_{1}(z_{0},z) & :=-\limfunc{Sc}Z^{(-1)}(1,z,z_{0})+j\limfunc{Sc}%
Z^{(-1)}(j,z,z_{0}), \\
w_{2}(z_{0},z) & :=\limfunc{Vec}Z^{(-1)}(1,z,z_{0})-j\limfunc{Vec}%
Z^{(-1)}(j,z,z_{0}).
\end{align*}
From (\ref{rela4}) and Proposition \ref{propmoreraadjunta} we conclude that $%
w_{1,2}(z_{0},z)$ are solutions of (\ref{vekuaadjoint}) in $\Omega
\backslash\{z_{0}\}$. On the other hand using (\ref{cauchykernelstrong}) we
see that the left-hand side of these equalities defines Cauchy kernels
corresponding to (\ref{vekuaadjoint}). Thus, (\ref{rela1}) and (\ref{rela2})
hold.

\item[$(ii)\Longrightarrow (iii)$] follows from Theorem \ref{SecondCIF}.

\item[$(iii)\Rightarrow (iv)$] and $(iv)\Rightarrow (i)$ is straightforward.
\end{description}
\end{proof}

In the following example we slightly change the Cauchy kernel from Remark %
\ref{RemExample} and obtain a reproducing but not global Cauchy kernel.

\begin{example}
Consider the Cauchy kernel 
\begin{equation*}
Z^{(-1)}(1,\zeta,z)=\frac{1}{z-\zeta}+\xi,\quad Z^{(-1)}(j,\zeta,z)=\frac {j%
}{z-\zeta}+\eta
\end{equation*}
corresponding to the equation $\partial_{\overline{z}}W=0$ with $\zeta
=\xi+j\eta$. Since%
\begin{equation}
-\limfunc{Sc}Z^{(-1)}(1,z,z_{0})+j\limfunc{Sc}Z^{(-1)}(j,z,z_{0})=\frac{1}{%
z-z_{0}}+z
\end{equation}
and%
\begin{equation}
\limfunc{Vec}Z^{(-1)}(1,z,z_{0})-j\limfunc{Vec}Z^{(-1)}(j,z,z_{0})=\frac{1}{%
z-z_{0}},
\end{equation}
due to Theorem \ref{Teo_rep_ker} $Z^{(-1)}(\alpha,\zeta,z)$ is a reproducing
Cauchy kernel in any bounded domain.

Let us prove that this kernel is not a restriction of any global Cauchy
kernel in any simply connected domain containing $(0,0)$ and $(1,0)$ as
interior points. Indeed, take $\zeta=1+\frac{1}{n}$ and $z=\zeta-\frac{1}{\xi%
}=1+\frac{1}{n}-\frac{n}{n+1}$. Then $Z^{(-1)}(1,\zeta,z)=0$. On the other
hand as we know a global Cauchy kernel can not vanish except at $z=\infty$.
\end{example}

\section{Construction of negative formal powers\label{Sect Negative Formal
Powers}}

Let $\left\{ \left( \widehat{F}_{n},\widehat{G}_{n}\right) \right\} $, $%
n=0,1,2\ldots$ be a generating sequence in a domain $\Sigma\subset \mathbb{C}%
_{j}$ embedding a generating pair corresponding to (\ref{vekuaadjoint}). In
this section we show how a set of negative formal powers corresponding to (%
\ref{vekua}) can be constructed following a simple algorithm when the
sequence $\left\{ \left( \widehat{F}_{n},\widehat{G}_{n}\right) \right\} $
and a reproducing Cauchy kernel for (\ref{vekua}) are known.

Let $Z^{(-1)}(\alpha,z_{0},z)$ be a reproducing Cauchy kernel of (\ref{vekua}%
) enjoying the properties from the hypothesis of Theorem \ref{Teo_rep_ker}.
Then, as was established in Theorem \ref{Teo_rep_ker}, $\widehat{Z}%
^{(-1)}(\alpha,z_{0},z)$ constructed by means of (\ref{rela1}) and (\ref%
{rela2}) is a reproducing Cauchy kernel for (\ref{vekuaadjoint}). For an
integer $n\geq2$ define%
\begin{equation*}
\widehat{Z}_{n-1}^{(-n)}(\alpha,z_{0},z):=\frac{(-1)^{n-1}}{(n-1)!}\frac{%
d_{\left( \widehat{F}_{n-2},\widehat{G}_{n-2}\right) }}{dz}...\frac{%
d_{\left( \widehat{F}_{1},\widehat{G}_{1}\right) }}{dz}\frac{d_{\left( 
\widehat{F}_{0},\widehat{G}_{0}\right) }}{dz}\widehat {Z}^{(-1)}(%
\alpha,z_{0},z)\text{.}
\end{equation*}
By construction, $\widehat{Z}_{n-1}^{(-n)}(\alpha,z_{0},z)$ is an $\left( 
\widehat{F}_{n-1},\widehat{G}_{n-1}\right) $-formal power of the order $-n$.
These are related with the formal powers for equation (\ref{vekua}) in the
following way.

\begin{theorem}
Under the above conditions suppose that the functions $\widehat{Z}%
_{n-1}^{(-n)}(\alpha ,z_{0},z)$, $\alpha =1,j$ are continuous in the
variable $z_{0}$ for fixed values of $z$ and that%
\begin{equation}
\lim_{z_{0}\rightarrow z}(z-z_{0})^{n}\widehat{Z}_{n-1}^{(-n)}(\alpha
,z_{0},z)=\alpha ,\quad \alpha =1,j\text{.}  \label{powersstrong}
\end{equation}%
Then 
\begin{equation}
Z^{(-n)}(1,z_{0},z):=(-1)^{n}\limfunc{Sc}\widehat{Z}%
_{n-1}^{(-n)}(1,z,z_{0})+j(-1)^{n+1}\limfunc{Sc}\widehat{Z}%
_{n-1}^{(-n)}(j,z,z_{0}),  \label{RelaFP1}
\end{equation}%
\begin{equation}
Z^{(-n)}(j,z_{0},z):=(-1)^{n+1}\limfunc{Vec}\widehat{Z}%
_{n-1}^{(-n)}(1,z,z_{0})+j(-1)^{n}\limfunc{Vec}\widehat{Z}%
_{n-1}^{(-n)}(j,z,z_{0})  \label{RelaFP2}
\end{equation}%
are negative formal powers corresponding to (\ref{vekua})$.$
\end{theorem}

\begin{proof}
Let $z_{0}\in\Sigma$ and $\Omega$ be a domain such that $z_{0}\notin 
\overline{\Omega}\subset\Sigma$. As mentioned above $\widehat{Z}%
^{(-1)}(\alpha,z_{0},z)$ constructed by means of (\ref{rela1}) and (\ref%
{rela2}) is a reproducing Cauchy kernel corresponding to (\ref{vekuaadjoint}%
). This in particular implies that 
\begin{equation*}
\int_{\partial\Omega}\widehat{Z}^{(-1)}(j\widehat{F}(\tau)d\tau,\tau
,z_{0})=\int_{\partial\Omega}\widehat{Z}^{(-1)}(j\widehat{G}(\tau)d\tau
,\tau,z_{0})=0.
\end{equation*}
Taking in the above equalities $n-1$ Bers' derivatives with respect to the
variable $z_{0}$ we obtain%
\begin{equation*}
\int_{\partial\Omega}\widehat{Z}_{n-1}^{(-n)}(j\widehat{F}(\tau)d\tau
,\tau,z_{0})=\int_{\partial\Omega}\widehat{Z}_{n-1}^{(-n)}(j\widehat{G}%
(\tau)d\tau,\tau,z_{0})=0
\end{equation*}
which equivalently can be written as follows%
\begin{equation}
\limfunc{Vec}\int_{\partial\Omega}\widehat{F}(\tau)w_{1,2}(z_{0},\tau)d\tau=%
\limfunc{Vec}\int_{\partial\Omega}\widehat{G}(\tau
)w_{1,2}(z_{0},\tau)d\tau=0,  \label{VecVec}
\end{equation}
where%
\begin{equation*}
\begin{array}{c}
w_{1}(z_{0},z):=-\limfunc{Sc}\widehat{Z}_{n-1}^{(-n)}(1,z,z_{0})+j\limfunc{Sc%
}\widehat{Z}_{n-1}^{(-n)}(j,z,z_{0}),\medskip \\ 
w_{2}(z_{0},z):=\limfunc{Vec}\widehat{Z}_{n-1}^{(-n)}(1,z,z_{0})-j\limfunc{%
Vec}\widehat{Z}_{n-1}^{(-n)}(j,z,z_{0}).%
\end{array}%
\end{equation*}

Equalities (\ref{VecVec}) together with Proposition \ref{propmoreraadjunta}
allow one to conclude that $w_{1,2}(z_{0},z)$ are solutions of (\ref{vekua})
in $\Omega\backslash\{z_{0}\}$. On the other hand, using (\ref{powersstrong}%
) we see that the left-hand side of (\ref{RelaFP1}) and (\ref{RelaFP2})
defines formal powers corresponding to (\ref{vekua}).
\end{proof}

\begin{remark}
The generating sequence $\left\{ \left( \widehat{F}_{n},\widehat{G}%
_{n}\right) \right\} $ required in the above construction can be obtained as
follows. Let $\left\{ \left( F_{-n},G_{-n}\right) \right\} $, $n=0,1,2\ldots$
be a generating sequence embedding a generating pair corresponding to (\ref%
{vekua}). Then by Propositions \ref{cara_coef_(F,G)*} and \ref%
{adjointsucessor}, $\left( \widehat{F},\widehat{G}\right)
:=(jF_{-1}^{\ast},jG_{-1}^{\ast})$ is a generating pair for (\ref%
{vekuaadjoint}) embedded in the generating sequence $\left\{ \left( \widehat{%
F}_{n},\widehat{G}_{n}\right) \right\} $, $n=0,1,2\ldots$ where $\left( 
\widehat{F}_{n},\widehat{G}_{n}\right) :=(j\left( F_{-n-1}\right)
^{\ast},j\left( G_{-n-1}\right) ^{\ast})$.
\end{remark}

\begin{remark}
Let us mention that if $\left\{ \left( F_{-n},G_{-n}\right) \right\} $, $%
n=0,1,2\ldots$ is a complete generating sequence of $\mathbb{C}_{j}$-valued
generating pairs then by Theorem \ref{theoremFPRelation} the corresponding $%
(F,G)$- and $\left( \widehat{F}_{n},\widehat{G}_{n}\right) $- negative
global formal powers satisfy (\ref{RelaFP1}) and (\ref{RelaFP2}).
\end{remark}

\section{Applications to the main Vekua and Schr\"{o}dinger equations\label%
{Sect Applications}}

Let us suppose that the Schr\"{o}dinger equation (\ref{sch}) has a
nonvanishing particular $\mathbb{C}_{i}$-valued solution $f\in C^{2}(\Omega)$%
. Then as it was explained in Section 3.3 the main Vekua equation (\ref%
{mainvekua}) is closely related to the Schr\"{o}dinger equation (\ref{sch}).
In this section we use this relation for obtaining the following results.
Starting from a fundamental solution of (\ref{sch}) we construct explicitly
a reproducing kernel and corresponding negative formal powers for (\ref%
{mainvekua}) as well as a fundamental solution for the Darboux transformed
equation (\ref{sch1}).

\begin{definition}
We say that a function $M(\zeta,z)$ satisfies condition I in $\Omega
\times\Omega$ if $M(\zeta,z)$ together with its first partial derivatives
corresponding to $\zeta$ are continuous functions in $\Omega\times
\Omega\backslash\left\{ \limfunc{diag}\right\} $, $M(\zeta,z)$ is bounded in
some neighborhood of every point $(z_{0},z_{0})$ and the same partial
derivatives behave like $\mathcal{O}(\log|z-\zeta|)$ as $(\zeta
,z)\rightarrow(z_{0},z_{0}).$
\end{definition}

We are mainly interested in fundamental solutions for (\ref{sch})%
\begin{equation*}
S(\zeta,z)=\log|z-\zeta|+R(\zeta,z)
\end{equation*}
fulfilling the following conditions:

\begin{description}
\item[$(C1)$] $S(z,\zeta)$ is a solution of (\ref{sch}) in both variables $z 
$ and $\zeta;$

\item[$(C2)$] $R(\zeta,z)\in C^{1}(\Omega\times\Omega\backslash\left\{ 
\limfunc{diag}\right\} )$ and $R$ is bounded in some neighborhood of every
point $(z_{0,}z_{0})$, $z_{0}\in\Omega$.

\item[$(C3)$] The functions $\partial_{x}R(\zeta,z)$, $\partial_{y}R(%
\zeta,z) $ satisfy the condition I in $\Omega\times\Omega$.
\end{description}

\subsection{Construction of a reproducing Cauchy kernel}

The purpose of this subsection is to construct a reproducing kernel for the
main Vekua equation (\ref{mainvekua}) from a fundamental solution of (\ref%
{sch}). First we notice the following fact.

\begin{proposition}
\label{prop_secessor=adjoint}Under the choice of the generating pair
corresponding to (\ref{mainvekua}) in the form (\ref{Gen pair for the main
Vekua}) the successor and the adjoint equation for (\ref{mainvekua})
coincide.
\end{proposition}

\begin{proof}
This is a direct consequence of Remark \ref{Remark_coeficients _main_vekua}.
\end{proof}

This observation together with Theorem \ref{Teo_rep_ker} allow us to state
that $Z^{(-1)}(\alpha,\zeta,z)$ is a reproducing kernel for (\ref{mainvekua}%
) when the formulas%
\begin{equation}
Z_{1}^{(-1)}(1,\zeta,z)=-\limfunc{Sc}Z^{(-1)}(1,z,\zeta )+j\limfunc{Sc}%
Z^{(-1)}(j,z,\zeta),  \label{Cauchykel1_MainVekua}
\end{equation}%
\begin{equation}
Z_{1}^{(-1)}(j,\zeta,z)=\limfunc{Vec}Z^{(-1)}(1,z,\zeta )-j\limfunc{Vec}%
Z^{(-1)}(j,z,\zeta).  \label{Cauchykel2_MainVekua}
\end{equation}
hold where $Z_{1}^{(-1)}(\alpha,\zeta,z)$ is some Cauchy kernel of (\ref%
{Main_vekua_suc}). We will start by constructing $Z_{1}^{(-1)}(1,\zeta,z)$
and then the above formulas can be used to obtain $Z_{1}^{(-1)}(j,\zeta,z)$
and $Z^{(-1)}(\alpha,\zeta,z)$.

\begin{proposition}
\label{kernel_coef_1}Let $S(\zeta,z)$ be a fundamental solution of (\ref{sch}%
) satisfying $(C1)-(C3)$. Then the function defined as follows 
\begin{equation}
Z_{1}^{(-1)}(1,\zeta,z):=2(\partial_{z}S(\zeta,z)-\frac{\partial_{z}f(z)}{%
f(z)}S(\zeta,z))  \label{ck1_sucess_constr}
\end{equation}
satisfies the following properties:

\begin{description}
\item[$(i)$] 
\begin{equation}
Z_{1}^{(-1)}(1,\zeta,z)=\frac{1}{z-\zeta}-2\frac{\partial_{z}f(z)}{f(z)}%
\log|z-\zeta|+M(\zeta,z),  \label{ck1_sucess_repre}
\end{equation}
where $M(\zeta,z)$ satisfies condition I.

\item[$(ii)$] $Z_{1}^{(-1)}(1,\zeta,z)$ is a Cauchy kernel of (\ref%
{Main_vekua_suc});

\item[$(iii)$] The scalar and the vector parts of $Z_{1}^{(-1)}(1,\zeta,z)$
are solutions of (\ref{sch}) in the variable $\zeta$.
\end{description}
\end{proposition}

\begin{proof}
$(i)$ Equality (\ref{ck1_sucess_repre}) and the properties of $M$ follow
directly from (\ref{ck1_sucess_constr}) and from the conditions $(C1)-(C3)$
defining the fundamental solution $S$.

$(ii)$ By construction $Z_{1}^{(-1)}(1,\zeta,z)$ is a solution of (\ref%
{Main_vekua_suc}) and from (\ref{ck1_sucess_repre}) we have%
\begin{equation*}
\underset{z\rightarrow\zeta}{\lim}(z-\zeta)Z_{1}^{(-1)}(1,\zeta,z)=1,
\end{equation*}
from which we conclude that this function is a Cauchy kernel.

$(iii)$ We prove that the scalar part of (\ref{ck1_sucess_constr})%
\begin{equation*}
\limfunc{Sc}Z_{1}^{(-1)}(1,\zeta,z)=\partial_{x}S(\zeta,z)-\frac {f_{x}(z)}{%
f(z)}S(\zeta,z)
\end{equation*}
is a solution of (\ref{sch}) in the variable $\zeta$. As $S(\zeta,z)$ is
already a solution of (\ref{sch}) with respect to $\zeta$, it is sufficient
to prove that $\partial_{x}S(\zeta,z)$ is a solution of the same equation in 
$\zeta$. We give a rigorous proof of this fact. Let $z_{0}=x_{0}+jy_{0}\in%
\Omega$ and define $\varphi(\zeta)=\partial_{x}S(\zeta,z_{0}),$%
\begin{equation*}
\varphi_{n}(\zeta)=\frac{S(\zeta,z_{0}+\frac{1}{n})-S(\zeta,z_{0})}{\frac {1%
}{n}},\quad n\in\mathbb{N}.
\end{equation*}
Each function $\varphi_{n}(\zeta)$ is a solution of (\ref{sch}) and $%
\varphi_{n}(\zeta)\rightarrow\varphi(\zeta)$ pointwise. Let $%
\zeta_{0}\in\Omega$, $\zeta_{0}\neq z_{0}$. We will prove that 
\begin{equation*}
\varphi_{n}(\zeta)\rightarrow\varphi(\zeta),\quad\partial_{\xi}\varphi
_{n}(\zeta)\rightarrow\partial_{\xi}\varphi(\zeta)\text{ and }\partial_{\eta
}\varphi_{n}(\zeta)\rightarrow\partial_{\eta}\varphi(\zeta)
\end{equation*}
uniformly is some neighborhood of $\zeta_{0}$. This together with Corollary %
\ref{conv_sch} allows us to conclude that $\varphi(\zeta)$ is a solution of (%
\ref{sch}) in such neighborhood of $\zeta_{0}$, and also in $\Omega
\backslash\left\{ z_{0}\right\} $ (because $\zeta_{0}$ is an arbitrary
point).

By the mean value theorem there exists some point $x_{n}\in \left[
x_{0,}x_{0}+\frac{1}{n}\right] $ such that%
\begin{equation*}
\varphi _{n}(\zeta )=\partial _{x}S(\zeta ,x_{n}+jy_{0})
\end{equation*}%
Let $\varepsilon >0$. Then by the continuity of $\partial _{x}S$ in $\Omega
\times \Omega \backslash \left\{ diag\right\} $ there exist $r>0$ and $%
n_{0}\in \mathbb{N}$ such that%
\begin{equation*}
\left\vert \varphi (\zeta )-\varphi _{n}(\zeta )\right\vert =\left\vert
\partial _{x}S(\zeta ,z_{0})-\partial _{x}S(\zeta ,x_{n}+jy_{0})\right\vert
<\varepsilon
\end{equation*}%
for all $\zeta $ belonging to the disk $D_{r}(\zeta _{0})$ of radius $r$ and
center $\zeta _{0}$ and $n\geq n_{0}$. This shows that $\varphi _{n}(\zeta
)\rightarrow \varphi (\zeta )$ uniformly in $D_{r}(\zeta _{0})$. Let us now
prove $\partial _{\xi }\varphi _{n}(\zeta )\rightarrow \partial _{\xi
}\varphi (\zeta )$. From the continuity $\ $of $\partial _{\xi }\partial
_{x}S$ in $\Omega \times \Omega \backslash \left\{ \limfunc{diag}\right\} $
and using the mean value theorem we have that there exists $x_{n}\in \left[
x_{0,}x_{0}+\frac{1}{n}\right] $ such that%
\begin{align*}
\partial _{\xi }\varphi _{n}(\zeta )& =\frac{\partial _{\xi }S(\zeta ,z_{0}+%
\frac{1}{n})-\partial _{\xi }S(\zeta ,z_{0})}{\frac{1}{n}} \\
& =\partial _{x}\partial _{\xi }S(\zeta ,x_{n}+iy_{0}) \\
& =\partial _{\xi }\partial _{x}S(\zeta ,x_{n}+iy_{0}).
\end{align*}%
As above, given $\varepsilon >0$ we can find some $r$ such that%
\begin{equation*}
\left\vert \partial _{\xi }\varphi _{n}(\zeta )-\partial _{\xi }\varphi
(\zeta )\right\vert =\left\vert \partial _{\xi }\partial _{x}S(\zeta
,x_{n}+iy_{0})-\partial _{\xi }\partial _{x}S(\zeta ,z_{0})\right\vert
<\varepsilon
\end{equation*}%
for all $\zeta \in D_{r}(\zeta _{0})$ which implies that $\partial _{\xi
}\varphi _{n}(\zeta )\rightarrow \partial _{\xi }\varphi (\zeta )$ uniformly
in $D_{r}(\zeta _{0})$. The corresponding convergence for the partial
derivatives with respect to $\eta $ is proved analogously.
\end{proof}

\bigskip

Formulas (\ref{Cauchykel1_MainVekua}), (\ref{Cauchykel2_MainVekua}) together
with $(iii)$ of the above proposition suggest that a Cauchy kernel of (\ref%
{Main_vekua_suc}) with the coefficient $\alpha=j$ can be constructed from (%
\ref{ck1_sucess_constr}) according to the formula%
\begin{equation}
Z_{1}^{(-1)}(j,\zeta,z):=T_{f(\zeta)}(-Z_{1}^{(-1)}(1,\zeta,z)).
\label{ckj_sucess_constr}
\end{equation}
where the integral operator $T_{f(\zeta)}$ (defined by (\ref{Darbouxttransf}%
)) acts with respect to the variable $\zeta$ along some path $\Gamma\subset
\Omega$ joining $\zeta_{0}$ with $\zeta$ and not passing through the point $%
z $. We prove this fact in the following proposition.

\begin{proposition}
\label{prop_ckj_const}Let $Z_{1}^{(-1)}(1,\zeta,z)$ be the Cauchy kernel of (%
\ref{Main_vekua_suc}) given by (\ref{ck1_sucess_constr}), $\zeta_{0}\in
\Omega$ and $\Omega_{0}$ be a simply connected subdomain of $\Omega$ such
that $\zeta_{0}\notin\Omega_{0}$. Then

\begin{description}
\item[$(i)$] the function (\ref{ckj_sucess_constr}) is a Cauchy kernel with
the coefficient $j$ for equation (\ref{ck1_sucess_constr}) in $\Omega_{0}$

\item[$(ii)$] the scalar and the vector parts of $Z_{1}^{(-1)}(j,\zeta,z)$
are solutions of (\ref{sch1}) in the variable $\zeta$

\item[$(iii)$] the equality holds%
\begin{equation}
Z_{1}^{(-1)}(j,\zeta,z)=\dfrac{j}{z-\zeta}+2j\dfrac{\partial_{z}f(z)}{f(z)}%
\log\left\vert z-\zeta\right\vert +N(\zeta,z)  \label{ckj_sucess_repre}
\end{equation}
where $N(\zeta,z)$ is a function satisfying condition I in $\Omega_{0}$.
\end{description}
\end{proposition}

\begin{proof}
It follows from $(iii)$ of Proposition \ref{kernel_coef_1} and from Remark %
\ref{T_f_def} that (\ref{ckj_sucess_constr}) is a well defined function when 
$\zeta$ belongs to any simply connected subdomain of $\Omega_{0}$ not
containing $z$. Since the function $M(\zeta,z)$ in (\ref{ck1_sucess_repre})
satisfies condition I, it is easy to see that%
\begin{equation*}
\lim_{\varepsilon\rightarrow0}T_{f(\zeta),\Gamma_{%
\varepsilon}}(-Z_{1}^{(-1)}(1,\zeta,z))=0
\end{equation*}
where $\Gamma_{\varepsilon}$ is the boundary of a disk with the center $z$
and radius $\varepsilon$. This implies that (\ref{ckj_sucess_constr}) is a
univalued function in $\Omega_{0}\backslash\left\{ z\right\} $. Then $(ii)$
of this proposition follows from Theorem \ref{Theo_Tf}.

Let us prove $(iii)$. It is clear that the function $N(\zeta ,z)$ in (\ref%
{ckj_sucess_repre}) together with its first partial derivatives
corresponding to $\zeta $ are continuous functions in $\Omega _{0}\times
\Omega _{0}\backslash \left\{ \limfunc{diag}\right\} $. It remains to study
the behavior of $N(\zeta ,z)$ and of its partial derivatives corresponding
to $\zeta $ when $(\zeta ,z)\rightarrow (z_{0},z_{0})\in \Omega _{0}\times
\Omega _{0}$. Consider the scalar part of $N(\zeta ,z)$, 
\begin{equation}
\limfunc{Sc}N(\zeta ,z)=T_{f(\zeta )}(-\limfunc{Sc}Z_{1}^{(-1)}(1,\zeta ,z))-%
\frac{y-\eta }{\left\vert z-\xi \right\vert ^{2}}-\frac{\partial _{y}f(z)}{%
f(z)}\log \left\vert z-\xi \right\vert .  \label{scN}
\end{equation}%
Let $z_{0}\in \Omega _{0}$, $\varepsilon _{0}>0$ such that $D_{2\varepsilon
_{0}}(z_{0})\subset \Omega _{0}$ and $\zeta =\xi +j\eta ,z=x+jy\in
D_{\varepsilon _{0}}(z_{0})$. Consider the path $\Gamma =\Gamma _{1}\cup
\Gamma _{2}\cup \Gamma _{3}$ where $\Gamma _{1}$ is some rectifiable curve
joining $\zeta _{0}$ with $z+\varepsilon _{0}$ and not passing through the
point $z$, and%
\begin{equation*}
\Gamma _{2}(t)=t+jy,\quad t\in \left[ x+\varepsilon _{0},x+\left\vert \zeta
-z\right\vert \right] ,
\end{equation*}%
\begin{equation*}
\Gamma _{3}(t)=z+\left\vert \zeta -z\right\vert e^{jt},\quad t\in \left[
0,\arg (\zeta -z)\right] .
\end{equation*}%
It is clear that $\Gamma $ leads from $\zeta _{0}$ to $\zeta $ and $\Gamma
_{2}\cup \Gamma _{3}\subset \Omega _{0}$. We have%
\begin{equation}
T_{f(\zeta ),\Gamma }(-\limfunc{Sc}Z_{1}^{(-1)}(1,\zeta
,z))=\sum_{m=1}^{3}T_{f(\zeta ),\Gamma _{m}}(-\limfunc{Sc}%
Z_{1}^{(-1)}(1,\zeta ,z)),  \label{Tf_gamam}
\end{equation}%
where%
\begin{equation*}
\limfunc{Sc}Z_{1}^{(-1)}(1,\zeta ,z)=\frac{x-\xi }{\left\vert z-\xi
\right\vert ^{2}}-\frac{f_{x}(z)}{f(z)}\log \left\vert z-\xi \right\vert +%
\limfunc{Sc}M(\zeta ,z).
\end{equation*}%
The function $T_{f(\zeta ),\Gamma _{1}}(-\limfunc{Sc}Z_{1}^{(-1)}(1,\zeta
,z))$ is continuous (since $z\notin \Gamma _{1}$) then and as $\limfunc{Sc}%
M(\zeta ,z)$ satisfies condition I, the function $T_{f(\zeta ),\Gamma
_{2}\cup \Gamma _{3}}(-\limfunc{Sc}M(\zeta ,z))$ is bounded in some
neighborhood of $(z_{0},z_{0}).$ The remaining terms of (\ref{Tf_gamam}) are

\begin{equation*}
\begin{array}{c}
\begin{array}{cc}
T_{f(\zeta),\Gamma_{2}}\left( \dfrac{\xi-x}{\left\vert z-\xi\right\vert ^{2}}%
+\dfrac{f_{x}(z)}{f(z)}\log\left\vert z-\xi\right\vert \right) & =\dfrac{%
f_{\eta}(z+\left\vert \zeta-z\right\vert )}{f(\zeta)}\log\left\vert
z-\xi\right\vert -\dfrac{f_{\eta}(z+\varepsilon_{0})}{f(\zeta)}\log
\varepsilon_{0}\medskip%
\end{array}
\\ 
+\dfrac{1}{f(\zeta)}\dint _{x+\varepsilon_{0}}^{x+\left\vert
z-\xi\right\vert }\left( \dfrac{f_{x}(z)}{f(z)}f_{\eta}(t+iy)-f_{\zeta%
\eta}(t+iy)\right) \log\left\vert t-x\right\vert dt,%
\end{array}%
\end{equation*}
\bigskip\medskip%
\begin{equation*}
T_{f(\zeta),\Gamma_{3}}\left( \dfrac{\xi-x}{\left\vert z-\xi\right\vert ^{2}}%
\right) =\dfrac{y-\eta}{\left\vert z-\xi\right\vert ^{2}}-\frac{1}{f(\zeta)}%
\dint _{0}^{\arg(\xi-z)}f_{\xi}(\Gamma_{3}(t))dt
\end{equation*}
\bigskip and%
\begin{equation*}
\begin{array}{c}
\begin{array}{cc}
T_{f(\zeta),\Gamma_{3}}\left( \log\left\vert z-\xi\right\vert \right) & =%
\dfrac{1}{f(\zeta)}\dint _{0}^{\arg(\xi-z)}f(\Gamma_{3}(t))dt%
\end{array}
\medskip \\ 
-\dfrac{\left\vert z-\zeta\right\vert \log\left\vert z-\zeta\right\vert }{%
f(\zeta)}\dint _{0}^{\arg(\xi-z)}\left( f_{\eta}(\Gamma_{3}(t))\sin
t+f_{\xi}(\Gamma_{3}(t)\cos t\right) dt.%
\end{array}%
\end{equation*}
Substituting these relations on the right-hand side of (\ref{scN}) we
conclude that the function $\limfunc{Sc}N(\zeta,z)$ is bounded when $%
(\zeta,z)\rightarrow(z_{0},z_{0})$. Differentiating (\ref{scN}) with respect
to $\xi$ and using the relation

\begin{equation*}
\begin{array}{r}
\begin{array}{cc}
\partial_{\xi}T_{f(\zeta)}(-\limfunc{Sc}Z_{1}^{(-1)}(1,\zeta,z))= & -\frac{%
f_{\xi}(\zeta)}{f(\zeta)}T_{f(\zeta)}(-\limfunc{Sc}Z_{1}^{(-1)}(1,\zeta,z))%
\end{array}
\medskip \\ 
-\frac{f_{\eta}(\zeta)}{f(\zeta)}\limfunc{Sc}Z_{1}^{(-1)}(1,\zeta
,z)+\partial_{\eta}\limfunc{Sc}Z_{1}^{(-1)}(1,\zeta,z)%
\end{array}%
\end{equation*}
we obtain that $\partial_{\xi}\limfunc{Sc}N(\zeta,z)=\mathcal{O}(\left\vert
\log|z-\zeta|\right\vert )$ as $(\zeta,z)\rightarrow(z_{0},z_{0})$. A
similar reasoning can be used to establish that $\partial_{\eta }\limfunc{Sc}%
N(\zeta,z)=\mathcal{O}(\left\vert \log|z-\zeta|\right\vert )$ as $%
(\zeta,z)\rightarrow(z_{0},z_{0})$. This shows that $\limfunc{Sc}N(\zeta,z)$
satisfies condition I in $\Omega_{0}$. An analogous reasoning is applicable
to $\limfunc{Vec}N(\zeta,z)$ that finishes the proof of part $(iii)$.

Finally, we prove that (\ref{ckj_sucess_constr}) is a solution of (\ref%
{Main_vekua_suc}) in the variable $z$. This together with (\ref%
{ckj_sucess_repre}) will imply $(i)$. Using the fact that $%
Z_{1}^{(-1)}(1,\zeta,z)$ is a solution of (\ref{Main_vekua_suc}) and with
the help of Theorem \ref{theorem_pseudo_conv} we can prove, reasoning as in
Proposition \ref{kernel_coef_1}, that $\partial_{\xi}Z_{1}^{(-1)}(1,\zeta,z)$%
, $\partial_{\eta}Z_{1}^{(-1)}(1,\zeta,z)$ are also solutions of (\ref%
{Main_vekua_suc}) in $z$. Then the fact that (\ref{ckj_sucess_constr}) is a
solution of (\ref{Main_vekua_suc}) is obtained from the integration with
respect to $\zeta$ of solutions of (\ref{Main_vekua_suc}) in the variable $z$%
. The proposition is proved.
\end{proof}

\begin{example}
\label{example_f=x}Let $f=x$. Consider the corresponding main Vekua equation 
$\partial _{\overline{z}}W=\frac{1}{2x}\overline{W}$ in some domain $\Omega $
such that $\overline{\Omega }$ has no common point with the axis $x=0$. The
successor equation has the form%
\begin{equation}
\partial _{\overline{z}}W=-\frac{1}{2x}\overline{W}  \label{ex1_suce_eq}
\end{equation}%
and the related Schr\"{o}dinger equations are $\Delta u=0$ \ and \ $\Delta v=%
\frac{2}{x^{2}}v$. A fundamental solution for the Laplace equation
satisfying $(C1)-(C3)$ can be chosen as $S(\zeta ,z)=\log \left\vert z-\zeta
\right\vert $. A reproducing Cauchy kernel for (\ref{ex1_suce_eq}) is
obtained by means of the procedure described above and has the following
form 
\begin{equation}
Z_{1}^{(-1)}(1,\zeta ,z)=\frac{1}{z-\zeta }-\frac{1}{x}\log \left\vert
z-\zeta \right\vert ,  \label{ex1_suce_CK1}
\end{equation}%
\begin{align*}
Z_{1}^{(-1)}(j,\zeta ,z)& =\frac{j}{z-\zeta }+\frac{y-\eta }{x\xi }(\log
\left\vert z-\zeta \right\vert -1)+\frac{j}{\xi }\log \left\vert z-\zeta
\right\vert \\
& -\frac{f(\zeta _{0})}{f(\zeta )}\left( \frac{j}{z-\zeta _{0}}+\frac{y-\eta
_{0}}{x\xi _{0}}(\log \left\vert z-\zeta _{0}\right\vert -1)+\frac{j}{\xi
_{0}}\log \left\vert z-\zeta _{0}\right\vert \right)
\end{align*}%
where $z=x+jy$, $\zeta =\xi +j\eta $ and $\zeta _{0}=\xi _{0}+j\eta _{0}\in
\Omega $ is some fixed point. Here the function inside the brackets is a
regular solution of (\ref{ex1_suce_eq}) for $z$ belonging to a domain not
containing $\zeta _{0}$. Therefore 
\begin{equation}
Z_{1}^{(-1)}(j,\zeta ,z)=\frac{j}{z-\zeta }+\frac{y-\eta }{x\xi }(\log
\left\vert z-\zeta \right\vert -1)+\frac{j}{\xi }\log \left\vert z-\zeta
\right\vert  \label{ex1_suce_CKj}
\end{equation}%
is a Cauchy kernel of (\ref{ex1_suce_eq}) as well. Using Theorem \ref%
{Teo_rep_ker} it is easy to verify that the expressions (\ref{ex1_suce_CK1})
and (\ref{ex1_suce_CKj}) also represent a reproducing Cauchy kernel for (\ref%
{ex1_suce_eq}). The corresponding reproducing Cauchy kernel for the main
Vekua equation such that (\ref{Cauchykel1_MainVekua}) and (\ref%
{Cauchykel2_MainVekua}) hold has the form%
\begin{equation}
Z^{(-1)}(1,\zeta ,z)=\frac{1}{z-\zeta }+\frac{1}{\xi }\log \left\vert
z-\zeta \right\vert -j\frac{y-\eta }{x\xi }(\log \left\vert z-\zeta
\right\vert -1),  \label{ex1_main_CK1}
\end{equation}%
\begin{equation}
Z^{(-1)}(j,\zeta ,z)=\frac{j}{z-\zeta }-j\frac{1}{x}\log \left\vert z-\zeta
\right\vert .  \label{ex1_main_CKj}
\end{equation}
\end{example}

\subsection{Construction of negative formal powers for main Vekua equations}

In this subsection we explain how a set of negative formal powers for
equation (\ref{mainvekua}) can be constructed from a fundamental solution of
(\ref{sch}) satisfying properties $(C1)-(C2)$. In the previous section a
reproducing Cauchy kernel for equation (\ref{mainvekua}) was constructed.
Then, as was shown in Section 5, using such Cauchy kernel and by means of
formulas (\ref{RelaFP1}), (\ref{RelaFP2}) a set of negative formal powers
for (\ref{mainvekua}) can be obtained whenever a generating sequence
embedding $(\widehat{F},\widehat{G})$ is known, where $(\widehat{F},\widehat{%
G})$ is a generating pair corresponding to the adjoint equation of (\ref%
{mainvekua}). Notice that $(f,\frac{j}{f})$ is a generating pair for (\ref%
{mainvekua}) and by Proposition \ref{prop_secessor=adjoint} $(\widehat{F},%
\widehat{G})$ is a successor of $(f,\frac{j}{f})$. Hence it is sufficient to
know a generating sequence embedding $(f,\frac{j}{f})$. As was shown in \cite%
{Bers Book} (see also \cite{APFT}) when $f$ has a separable form $%
f=\phi(x)\psi(y)$ where $\phi$ and $\psi$ are arbitrary twice continuously
differentiable functions, there exists a periodic generating sequence with a
period two in which $(f,\frac{j}{f})$ is embedded,%
\begin{equation*}
(F,G)=\left( \phi\psi,\frac{\,j}{\phi\psi}\right) ,\quad(F_{1},G_{1})=\left( 
\frac{\psi}{\phi},\,\frac{j\phi}{\psi}\right) ,\quad (F_{2},G_{2})=\left(
F,G\right) ,\quad(F_{3},G_{3})=(F_{1},G_{1}),\ldots
\end{equation*}
(methods for construction of generating sequences in more general situations
are discussed in \cite{KrRecDev}, \cite{APFT} and \cite{Tranplante Op}).

\begin{example}
Consider $f=x$ and the corresponding main Vekua equation $\partial _{%
\overline{z}}W=\frac{1}{2x}\overline{W}$. The negative formal powers
constructed according to the described above procedure beginning with the
reproducing Cauchy kernel (\ref{ex1_main_CK1}), (\ref{ex1_main_CKj}) have
the form%
\begin{equation*}
Z^{(-2)}(1,\zeta,z)=\frac{1}{\left( z-\zeta\right) ^{2}}+\frac{j}{x}\limfunc{%
Vec}\frac{1}{z-\zeta},
\end{equation*}%
\begin{equation*}
Z^{(-2)}(j,\zeta,z)=\frac{j}{\left( z-\zeta\right) ^{2}}-\frac{1}{\xi}\frac{j%
}{z-\zeta}+\frac{j}{x}\left( \limfunc{Vec}\frac{1}{z-\zeta }+\frac{1}{\xi}%
\log\left\vert z-\zeta\right\vert \right) ,
\end{equation*}
for an odd $n\geq3$ 
\begin{equation*}
Z^{(-n)}(1,\zeta,z)=\frac{1}{\left( z-\zeta\right) ^{n}}-\frac{1}{(n-1)\xi }%
\frac{1}{\left( z-\zeta\right) ^{n-1}}-\frac{j}{(n-1)x}\limfunc{Sc}\left( 
\frac{j}{\left( z-\zeta\right) ^{n-1}}-\frac{1}{(n-2)\xi}\frac {j}{\left(
z-\zeta\right) ^{n-2}}\right) ,
\end{equation*}%
\begin{equation*}
Z^{(-n)}(j,\zeta,z)=\frac{j}{\left( z-\zeta\right) ^{n}}-\frac{j}{(n-1)x}%
\limfunc{Sc}\frac{1}{\left( z-\zeta\right) ^{n-1}}
\end{equation*}
and for an even $n\geq4$ 
\begin{equation*}
Z^{(-n)}(1,\zeta,z)=\frac{1}{\left( z-\zeta\right) ^{n}}+\frac{j}{(n-1)x}%
\limfunc{Vec}\frac{1}{\left( z-\zeta\right) ^{n-1}},
\end{equation*}%
\begin{equation*}
Z^{(-n)}(j,\zeta,z)=\frac{j}{\left( z-\zeta\right) ^{n}}-\frac{1}{(n-1)\xi }%
\frac{j}{\left( z-\zeta\right) ^{n-1}}+\frac{j}{(n-1)x}\limfunc{Vec}\left( 
\frac{j}{\left( z-\zeta\right) ^{n-1}}-\frac{1}{(n-2)\xi}\frac {j}{\left(
z-\zeta\right) ^{n-2}}\right) .
\end{equation*}
\end{example}

\subsection{Construction of fundamental solutions using the Darboux-type
transformation}

Let $f\in C^{2}(\Omega )$ be a nonvanishing particular solution of (\ref{sch}%
). Consider the operator $T_{f}$ defined by (\ref{Darbouxttransf}) that
transforms regular solutions of (\ref{sch}) into regular solutions of the
Darboux transformed equation (\ref{sch1}). We stress that direct application
of $T_{f}$ to a fundamental solution of (\ref{sch}) does not lead to a
fundamental solution of (\ref{sch1}). The result rather should be a
multivalued solution of (\ref{sch1}) with the behaviour similar to that of $%
\arg (z-\zeta )$ (the imaginary part of the function $\ln (z-\zeta )$). In
this subsection we describe a procedure to construct a fundamental solution
of (\ref{sch1}) from a known fundamental solution of (\ref{sch}).

Let $S(\zeta ,z)$ be a fundamental solution of (\ref{sch}) satisfying the
conditions $(C1)-(C3)$. Using formula (\ref{ck1_sucess_constr}) one can
construct $Z_{1}^{(-1)}(1,\zeta ,z)$ and use it to obtain $%
Z_{1}^{(-1)}(j,\zeta ,z)$ by means of (\ref{ckj_sucess_constr}). Then a
fundamental solution of (\ref{sch1}) is constructed as follows.

\begin{proposition}
Let $Z_{1}^{(-1)}(j,\zeta,z)$ be the Cauchy kernel given by (\ref%
{ckj_sucess_constr}) and characterized by Proposition \ref{prop_ckj_const}.
Then%
\begin{equation}
\begin{array}{rl}
S_{1}(z,\zeta) & :=Vec\dint _{z_{0}}^{z}Z_{1}^{(-1)}(j,\zeta,\tau)d_{(f,%
\frac{i}{f})}\tau\medskip \\ 
& =\frac{1}{f(z)}Vec\dint _{z_{0}}^{z}f(\tau)Z_{1}^{(-1)}(j,\zeta,\tau)d\tau%
\end{array}
\label{sf2_contr}
\end{equation}
is a fundamental solution of (\ref{sch1}) in $\Omega_{0}$ enjoying
properties $(C1)-(C3)$. Here $z_{0}$ is some fixed point different from $%
\zeta_{0}$ and such that $z_{0}\in\Omega\backslash\overline{\Omega_{0}}$.
\end{proposition}

\begin{proof}
By construction the integral on the right-hand side of (\ref{sf2_contr})
does not depend on the path joining $z_{0}$ with $z$ (obviously not passing
through the point $\zeta$). Substituting into (\ref{sf2_contr}) the
representation (\ref{ckj_sucess_repre}) of $Z_{1}^{(-1)}(j,\zeta,\tau)$ we
obtain

\begin{equation*}
\begin{array}{ll}
S_{1}(z,\zeta) & =\dfrac{1}{f(z)}Vec\dint _{z_{\ast}}^{z}f(\tau)\left( 
\dfrac{j}{\tau-\zeta}+2j\dfrac{\partial_{\tau}f(\tau)}{f(\tau )}%
\log\left\vert \tau-\zeta\right\vert +N(\zeta,\tau)\right) d\tau\medskip \\ 
& =\log\left\vert z-\zeta\right\vert -\dfrac{f(z_{0})}{f(z)}\log\left\vert
z_{0}-\zeta\right\vert +\dfrac{1}{f(z)}Vec\dint
_{z_{0}}^{z}f(\tau)N(\zeta,\tau)d\tau.%
\end{array}%
\end{equation*}
From the last equality and using the fact that $N(\zeta,\tau)$ satisfies
condition I we conclude that $S_{1}(z,\zeta)$ is indeed a fundamental
solution of (\ref{sch1}). The remaining properties follow directly from the
properties of $Z_{1}^{(-1)}(j,\zeta,\tau)$ established in Proposition \ref%
{prop_ckj_const}.
\end{proof}

\begin{example}
Let us consider the case described in Example \ref{example_f=x}. Using the
kernel (\ref{ex1_suce_CKj}) and formula (\ref{sf2_contr}) with $%
z_{0}=\zeta+1 $ we obtain the following fundamental solution for the
operator $\Delta-\frac {2}{x^{2}}I$%
\begin{equation*}
S_{1}(z,\zeta)=\log\left\vert z-\zeta\right\vert +\frac{\left\vert
z-\zeta\right\vert ^{2}}{2x\xi}\log\left\vert z-\zeta\right\vert -\frac{%
\left\vert z-\zeta\right\vert ^{2}+2(y-\eta)^{2}-1}{4x\xi}.
\end{equation*}
\end{example}

\begin{remark}
Since the fundamental solution (\ref{sf2_contr}) possesses the properties $%
(C1)-(C3)$ one can apply to it the above procedure and construct fundamental
solutions of new Schr\"{o}dinger equations obtained from (\ref{sch1}) by
applying further Darboux transformations. The procedure can be repeated a
finite number of times. In this way it is possible to obtain fundamental
solutions of several Schr\"{o}dinger equations in a closed form as well as
the negative formal powers for the corresponding main Vekua equations.
\end{remark}

\end{document}